\documentclass{article}
\usepackage[utf8]{inputenc}
\usepackage{amsmath, amsthm, amssymb, enumitem, mathtools}
\usepackage{authblk}
\usepackage{mathrsfs}
\usepackage{tikz-cd}
\usepackage[
	pdftitle={Degeneration of Sheaves on Fibered Surfaces},
	pdfauthor={Nikolas Kuhn},]
{hyperref}

 \setlist[enumerate]{leftmargin=4em}
\newtheorem{thm}{Theorem}[section]
\newtheorem{cor}[thm]{Corollary}
\newtheorem{lem}[thm]{Lemma}
\newtheorem{prop}[thm]{Proposition}

\theoremstyle{definition}
\newtheorem{definition}[thm]{Definition}
\newtheorem{rem}[thm]{Remark}

\newtheorem{assu}[thm]{Assumption}

\renewcommand{\AA}{{\mathbb{A}}}
\newcommand{\CC}{{\mathbb{C}}}
\newcommand{\PP}{{\mathbb{P}}}
\newcommand{\QQ}{{\mathbb{Q}}}
\newcommand{\ZZ}{{\mathbb{Z}}}

\newcommand{\CE}{{\mathcal{E}}}

\newcommand{\CM}{{\mathcal{M}}}

\newcommand{\vir}{{\mathrm{vir}}}

\DeclareMathOperator{\Hom}{Hom}
\DeclareMathOperator{\Ext}{Ext}
\DeclareMathOperator{\Exp}{Exp}
\DeclareMathOperator{\Ell}{Ell}
\DeclareMathOperator{\Spec}{Spec}
\DeclareMathOperator{\ch}{ch}

\DeclareMathOperator{\Pic}{Pic}

\DeclareMathOperator{\undeg}{\underline{deg}}

\DeclarePairedDelimiter\rough{\lfloor}{\rceil}

\newcommand{\Xtilde}{\widetilde{X}}
\newcommand{\Ytilde}{\widetilde{Y}}
\newcommand{\Ctilde}{\widetilde{C}}
\newcommand{\Ztilde}{\widetilde{Z}}

\title{Degeneration of Sheaves on Fibered Surfaces}
\author{Nikolas Kuhn}
\date{May 2023}

\begin{document}

\maketitle

\begin{abstract}
We construct moduli stacks of stable sheaves for surfaces fibered over marked nodal curves by using expanded degenerations. These moduli stacks carry a virtual class and therefore give rise to enumerative invariants. In the case of a surface with two irreducible components glued along a smooth divisor, we prove a degeneration formula that relates the moduli space associated to the surface with the relative spaces associated to the two components. For a smooth surface and no markings, our notion of stability agrees with slope stability with respect to a suitable choice of polarization. We apply our results to compute elliptic genera of moduli spaces of stable sheaves on some elliptic surfaces. 
\end{abstract}

\section{Introduction}
Let $X_0 = Y_1\cup_D Y_2$ be a projective surface that is the union of smooth surfaces $Y_i$ along a common smooth divisor $D$. In \cite[\S 1, \S 4]{Don_Flo} Donaldson raises the problem of constructing a good theory of stable sheaves on such an $X_0$, with the following properties. 
\begin{enumerate}[label = \Roman*.]
	\item It behaves well under smoothings. In other words, the numerical invariants of the theory on $X_0$ agree with the invariants of the usual moduli space of Gieseker-stable sheaves on a smoothing of $X_0$ when those are defined.
	\item It behaves well under decomposition. More precisely, there should be spaces of ``relative stable sheaves'' for a pair $(Y,D)$ of a smooth surface $Y$ with a smooth divisor $D$, such that the moduli space of sheaves on $X_0$ can be related to the relative spaces for the pairs $(Y_i,D)$.
\end{enumerate}

Such a theory would enable one to compute sheaf-theoretic invariants of projective surfaces through degenerations to a reducible surface. This has been successfully implemented in other settings -- notably in Gromov--Witten theory (\cite{Li_Stab}, \cite{Li_GWDeg}, \cite{ACGS_Dec}, \cite{Ran_Log}) and Donaldson--Thomas theory (see for example \cite{LiWu_Good}, \cite{MaRa_Log}). 

In this paper we answer Donaldson's questions for fibered surfaces. Let $f:X \to C$ be a surface fibered over a marked nodal curve (see Definition \ref{def_fib} for the precise meaning). Then we propose the following definition of "stability on the fiber".
\begin{definition}[cf. Definition \ref{deffstable}]
	A coherent sheaf $E$ on $X$ is $f$-stable, if it is torsion-free and 
	\begin{enumerate}[label = \roman*)]
		\item for every node or marked point $x\in C$, the restriction of $E$ to $f^{-1}(x)$ is a slope-stable vector bundle, and 
		\item for every generic point $\eta$ of an irreducible component of $C$, the restriction of $E$ to $f^{-1}(\eta)$ is slope stable.
	\end{enumerate}
\end{definition}

Our first main result is that this notion of stability leads to well-behaved moduli spaces. This gives an answer to I for fibered surfaces.

\begin{thm}[cf. Theorem \ref{thm_properdm} and Proposition \ref{prop_virclass}]\label{thm_mainthm}
 Let $d$ and $r>0$ be coprime integers. Let $c_1\in H^2(X,\ZZ)$ be a cohomology class with $c_1\cap f^{-1}(x) = d[pt]$ for any $x\in C$. Let $\alpha$ be a generic stability condition on $C$. Then for any $\Delta\in \ZZ$, there exists a proper Deligne--Mumford moduli stack
 \[M^{\alpha}_{X/C}(r,c_1,\Delta)\] 
 parametrizing $f$-stable, $\alpha$-balanced sheaves of rank $r$, first Chern class $c_1$ and discriminant $\Delta$ on expansions of $X$. Moreover, this stack has a natural virtual fundamental class and the numerical invariants are invariant under deformations of $X$ together with the fibration and choice of $c_1$.
\end{thm}

Here we introduced two additional subtleties: Expansions of a degenerate surface (cf. Definition \ref{defexpansion}) and balancing with respect to a stability condition $\alpha$ on a curve (cf. \S \ref{subsec_balanced}).

Our construction automatically yields a notion of relative moduli space: Putting markings $y_1,\ldots,y_n$ on $C$ corresponds to working relative to the divisor $f^{-1}(y_1)\cup \cdots \cup f^{-1}(y_n)$. 

Our second main result is a proof of II in the case that $X=Y_1\cup_F Y_2$ is a union of two irreducible components which meet along a fiber $F$ of $f$, and so that we also have $C = C_1\cup_x C_2$. We state this result somewhat informally (see Theorem \ref{thm_decform} and Proposition \ref{prop_fixdetfundclass} for more precise versions).

\begin{thm}\label{thm_decomp}
	The enumerative invariants of moduli spaces of $f$-stable sheaves on $X$ can be recovered from the relative invariants associated to the pairs $(Y_i, F)$. 
\end{thm}

For elliptic surfaces, the relative invariants can be determined from the absolute ones, which makes Theorem \ref{thm_decomp} more powerful in practice. To illustrate this, we give an application to elliptic genera.

For a proper scheme $M$ with perfect obstruction theory, let $\Ell^{\vir}(M)$ denote its virtual elliptic genus, as defined for example in the introduction of \cite{GK_Rank2}. We let also $\phi_{0,1}(q,y)$ be the weak Jacobi form and $\mathbf{L}(\phi_{0,1},p)$ its Borcherds lift as presented there. We recall that if $f:X\to \mathbb{P}^1$ is an elliptic surface, its \emph{degree} is the degree of the line bundle $(R^1f_*\mathcal{O}_X)^{\vee}$.

\begin{thm}\label{thm_ellgen}
Let $X$ be a degree $e\geq 2$ elliptic surface over $\mathbb{P}^1$ without multiple or reducible fibers. Let $d\geq 1$ be minimal so that $X$ has a $d$-section with divisor class $D$. Let $H$ be an ample line bundle on $X$ and $[F]$ the cohomology class of a fiber. Let $M_{X,H}(r, D,\Delta)$ denote the moduli space of $H$-Gieseker-stable sheaves on $X$ of rank $r>0$, first Chern class $D$ and discriminant $\Delta$. Assume that $d$ is coprime to $r$, and that $H$ is chosen so that stability equals semistability for sheaves of rank $r$ and first Chern class $c_1$, and all $\Delta$. Consider the generating series 
\[Z^{\Ell}_{X,r,c_1}(p):= \sum_{\substack{\Delta\in \ZZ \\ 0\leq \ell < r}} \Ell^{\vir}(M_{X,H}(r,c_1+\ell [F] ,\Delta))\, p^{\dim M_{X,H}(r,c_1+\ell [F] ,\Delta)}.\]
Then  
\[Z^{\Ell}_{X,r,c_1}(p) = \sum_{n\geq 0} \Ell(\operatorname{Hilb}_n(X))\, p^{2n} = \left(\frac{1}{\mathbf{L}(\phi_{0,1},p^2)}\right)^{\chi(\mathcal{O}_X)}.\]
\end{thm}  
In the case $e=2$ the surface $X$ is a K3-surface, and the statement reduces to the rank $1$ case, which was stated in \cite{DMVV} and has been proven in \cite{BoLi_Ell}, \cite{BoLi_McKay}.
\begin{rem}
	The proof of Theorem \ref{thm_ellgen} goes through with ``virtual elliptic genus'' replaced by virtual cobordism to give
	\[Z^{\operatorname{cob}}_{X,r,c_1}(p) = \left(\sum_{n\geq 0}[K3^{[n]}] p^{2n}\right)^{\chi(\mathcal{O}_X)/2},\]
	where the possible half-integer power is chosen to have constant coefficient one.  
	In particular, this answers Göttsche and Kool's Conjecture 7.7 in \cite{GK_Rank2} affirmatively for this class of elliptic surfaces.
\end{rem}

\paragraph{Background}

Our approach is inspired by earlier constructions of Gieseker--Li \cite{GiLi_Irr} and Li \cite{Li_Mod}, \cite{Li_Tow}: They consider moduli spaces of Simpson semi-stable sheaves on expansions of a degenerate surface $X_0$.
 Under good conditions, this does result in moduli spaces which behave well under deformations and carry a virtual fundamental class. However, their approach does not lead to a degeneration formula, since it is unclear how to relate the moduli space of sheaves on the degeneration with the relative spaces for the pairs $(Y_i,D)$. The problem is that Simpson stability of a sheaf on $X_0$ can not be determined by only looking at its restrictions to the $Y_i$, but also depends on how the sheaves are glued in a subtle manner. 
 
By restricting to fibered surfaces, we get around this issue: The notion of $f$-stability is defined in terms of restriction on fibers, which can be checked on each component separately. The use of expansions is crucial for us: As in \cite{GiLi_Irr}, it allows us to work only with sheaves that are locally free along the singular locus of $X_0$. Just as importantly, we can guarantee stability on the fibers over marked points and nodes, which gives us restriction maps to the moduli spaces of stable vector bundles on curves. This is central to the decomposition result of Theorem \ref{thm_decomp}. 

Since $f$-stability on $X\to C$ is unchanged with respect to tensoring by a line bundle pulled back from $C$, the moduli space of $f$-stable sheaves will not be separated in families when $C$ degenerates from a smooth to a reducible curve, since the Picard scheme of $C$ is not separated. Essentially, the issue is that in a one-parameter family of curves, the limit of the trivial line bundle doesn't need to be trivial, since one can twist by components of the special fiber. To deal with this issue, we need to restrict the possible twists of an $f$-stable sheaf. This is achieved by picking a stability condition $\alpha$ on $C$ when the base curve becomes reducible, and demanding a certain numerical balancing condition which singles out a unique choice of twist. 

Although $f$-stability at first seems unrelated to Gieseker or slope stability, it turns out that these notions agree on a smooth fibered surface, when one considers the latter with respect to a suitable choice of polarization. This is already present in the work of Yoshioka and is included here as Theorem \ref{thm_boundedwalls}. Moreover, due to results of Mochizuki \cite{Mo_Don}, for surfaces with $p_g(X)>0$, virtual enumerative invariants are independent of choice of polarization whenever stability equals semistability. Thus, for such surfaces, there is no loss of generality in considering $f$-stability for the computation of invariants.  
 
\paragraph{Structure of the paper.}
In \S 3, we introduce the definitions and constructions that go into Theorem \ref{thm_mainthm}. Since we work over a general base $B$, we automatically obtain deformation invariance. 

In \S 4, we show the decomposition in the situation of Theorem \ref{thm_decomp}. We include the proof of Theorem \ref{thm_ellgen} in \S \ref{subsec_ellfib}. 

In \S 2, we collect some material that is needed for the main constructions in the later sections. For a first reading, we suggest only taking a look at Definition \ref{defexpansion} and Lemma \ref{lem_expprop} in \S 2.1, and otherwise to refer to this section only as needed. 

\paragraph{Relation to other work.}
Since the modern mathematical definition of Vafa--Witten invariants on algebraic surfaces by Tanaka--Thomas (\cite{TT_VW1}, \cite{TT_VW2}), there has been renewed interest in the enumerative geometry of moduli spaces of sheaves on surfaces. 
Göttsche and Kool, in a series of work, developed many conjectures for the structure of such invariants (see \cite{GK_Sheaves} for an excellent overview). However, beyond the cases of Hilbert schemes and special classes of surfaces such as rational, elliptic or K3 surfaces, very little is known -- notable exceptions are results on Donaldson invariants (\cite{GNY}, \cite{GNY_SW}) and Blowup formulas (\cite{KT_Blowup}, \cite{KLT_Blowup}).

Recently, Dominic Joyce has announced results regarding deep structure theorems for enumerative invariants of surfaces with $p_g>0$, building on his theory of wall-crossing in abelian categories and his version of Mochizuki's rank reduction algorithm \cite{Joyce_En}. His result shows that the generating series of invariants are determined in terms of a number of universal power series and universal constants, and of finitely many fundamental enumerative invariants of the surface. We hope that the current work can be used to exhibit relations between, and therefore help determine his universal series.   

For rank one sheaves, our results reduce to a theory of Hilbert schemes on degenerations, which is treated in \cite{LiWu_Good} in the case of a smooth singular locus, and was later generalized to arbitrary normal crossings degenerations in \cite{MaRa_Log}. 

\paragraph{Further directions}
One problem that is not addressed here is to generalize the results of \S 4 to the case of a non-separating node, i.e. given a fibered surface $X\to C$ and a non-separating node $x$ of $C$, to describe the enumerative invariants for the moduli space of $f$-stable sheaves on $X\to C$ in terms of those on $X'\to C'$, where $C'\to C$ is the partial normalization of $C$ at $x$, and $X'=X\times_C C'$. This should be possible, but requires a closer analysis of the combinatorics of the Picard scheme of $C$. An application of this, suggested by Jørgen Rennemo, would be to obtain a $(1+1)$-dimensional cohomological field theory by considering surfaces of the form $F\times C \to C$, where $F$ is kept fixed and $C$ is an arbitrary marked nodal curve. 

In another direction, the constructions here should generalize beyond the case of surfaces, which we mostly require to obtain a good enumerative theory. We expect that the same methods generalize for example to give degeneration formulas for fibered Fano and Calabi-Yau threefolds. We thank Richard Thomas for pointing this out to us.     

\paragraph{Acknowledgements} The author would like to thank Nicola Pagani, John-Christian Ottem and Richard Thomas for helpful discussions during the writing of this paper. Special thanks goes to Jørgen Rennemo for suggesting this topic and for regular discussions. This research is funded by Research Council of Norway grant number 302277 - "Orthogonal gauge duality and non-commutative geometry".

\paragraph{Notations and Conventions.}

\begin{itemize}

\item All schemes and stacks we consider will be locally Noetherian over $\CC$.

\item By a curve over a base $B$, we mean a flat and finite type algebraic space over $B$ with one-dimensional fibers. We drop reference to the base when $B=\Spec \CC$.

\item By a marked nodal curve over $B$ we mean a curve over $B$ with at worst nodal singularities and a finite number of sections which are disjoint and do not meet the singularities. Here, by ``at worst nodal'', we mean \'etale locally of the form $Z(xy-f)\subseteq B\times \AA^2$, where $x,y$ are standard coordinates on $\AA^2$ and $f$ is a function on $B$.

\item Unless noted otherwise, we will assume all curves to be proper over the base and to have geometrically connected fibers. 

\item For (numerical) divisor classes $D_1,D_2$ on a proper algebraic surface $X$, and more generally for classes in the second cohomology of $X$, we denote by $(D_1,D_2) \in \mathbb{Q}$ their intersection product. 
\item We will often indicate the base (resp. base changes) of a family of curves  or surfaces simply by a subscript (resp. by the change thereof). 
\end{itemize}

In this paper, we will use the following notion of fibered surface.
\begin{definition}\label{def_fib}
 We say $f:X\to C$ is a fibered surface, if $C$ is a marked nodal curve, and 
	\begin{enumerate}[label = \roman*)]
		\item $f$ is flat and proper of dimension one with geometrically connected fibers,
		\item $f$ is smooth over the nodes and marked points of $C$,
		\item for every irreducible component $D\subseteq C$, the scheme-theoretic pre-image $f^{-1}(D)\subseteq X$ is a smooth projective surface.  
	\end{enumerate}
We say $f:X_B \to C_B$ is a family of fibered surfaces over a base $B$, if $C_B$ is a marked nodal curve over $B$ and $f$ is flat and proper, such that $X_b\to C_b$ is a fibered surface for every geometric point $b$ of $B$. 
\end{definition}

\section{Preliminaries}\label{sec_prelim}

\subsection{Expansions and Expanded Degenerations}\label{subsec_exp}
We define what we mean by a family of expansions of marked nodal curves and show some basic properties
Let $(C_B,\sigma_1,\ldots,\sigma_n)$ be a flat family of at most nodal marked curves (not necessarily assumed proper or connected) over a base $B$.  

\begin{definition}\label{defexpansion}
	Let $T$ be a scheme. A family of \emph{expansions of $C_B$ over $B$} parametrized by $T$ is given by a morphism $b:T\to B$, a flat family of nodal marked curves $(\Ctilde_T,\widetilde{\sigma}_1,\ldots,\widetilde{\sigma}_n)$ over $T$ and proper morphism of marked curves $c:\Ctilde_T \to C_T$ such that
	
	\begin{enumerate}[label = \roman*)]
		\item the natural map $\mathcal{O}_{C_T}\to Rc_*\mathcal{O}_{\Ctilde_T}$ is an isomorphism,
		\item for each $t\in T$, we have a fiberwise isomorphism of twisted dualizing sheaves $c^*\omega_{C_t}(\sigma_1,\ldots,\sigma_n) \equiv \omega_{\Ctilde_t}(\widetilde{\sigma}_1,\ldots,\widetilde{\sigma}_n)$.
	\end{enumerate}
	
	We let $\Exp_{C/B}$ denote the stack parametrizing expansions over $B$. 
\end{definition}

One can give an alternative more explicit characterization:
\begin{lem}\label{lem_expprop}
	Let $\Ctilde_T$ and $C_T$ be flat families of nodal marked curves over a base $T$ and let $c:\Ctilde_T\to C_T$ be a morphism of marked curves over $B$. Then condition i) of Definition \ref{defexpansion} is equivalent to
		
	\begin{enumerate}[label = \roman*')]
		\item The locus $\Sigma_c\subset C_T$ where the map $c$ is not an isomorphism is quasi-finite over $T$. For each $x \in \Sigma_c$, the scheme-theoretic fiber $c^{-1}(x)$ is a connected nodal curve of arithmetic genus zero.   
	\end{enumerate}
Assuming this condition holds, then ii) of Definition \ref{defexpansion} is equivalent to

	\begin{enumerate}[label = \roman*')]
		\setcounter{enumi}{1}
		\item If $x\in\Sigma_c$ is lying over $y\in T$, then $x$ is either a node or a marked point in the fiber $C_y\subset C_T$ and  $c^{-1}(x)$ is a chain of rational curves $R_1\cup \cdots \cup R_j$, containing two distinguished points of $\Ctilde_y$, one of which lies on $R_1$ and one on $R_j$ (these are two nodes if $x$ is a node, or one node and one marked points if $x$ is a marked point).
	\end{enumerate}	
	
\end{lem}
\begin{proof}
	We address the first part. By properness, $c$ is finite over the locus where $c$ has zero-dimensional fibers.  Then condition i) holds over this locus, if and only if $c$ is an isomorphism there. 
	Now let $x\in C_T$ be a point where $c$ has a one-dimensional fiber. It follows that $c$ has to be a contraction of components. Since formation of $R^1c_*$ commutes with base change,  vanishing of $R^1c_*\mathcal{O}_{\Ctilde}$ around $x$ is equivalent to $c^{-1}(x)$ being of arithmetic genus zero. Taken together, this establishes that i) and i') are equivalent. 
		
	We address the second part. Assume i) and i') hold. If ii) holds, then the twisted dualizing sheaf must be trivial on any components contracted by $c$. We already know each such component is rational, so they must have precisely two distinguished points. In particular, they must be arranged as a chain. The distinguished points on the end of the chain must come from intersection with a remaining component of $\Ctilde$ or from marked points. If there is a marked point, it must map to a marked point of $C$ (in particular, there can be only one). Otherwise, the two points of intersection must come from the intersection with two branches of a node on $C$.
	
	Conversely, we have that ii') implies ii) by direct computation.
\end{proof}

\begin{lem}
Let $\Ctilde_T$ and $C_T$ be flat and proper families of nodal marked curves over a base $T$ and let $c:\Ctilde_T\to C_T$ be a morphism of marked curves over $T$. Then condition i) is an open condition. On the locus where i) holds, ii) is an open condition.
\end{lem}
\begin{proof}
Openness of i) is straightforward. Assume that i) holds for $c$.
Then condition ii') is equivalent to vanishing of $R^1c_*(\omega_{\Ctilde_B/B}(\sigma_1'+\ldots + \sigma_n'))^{\otimes 2 }$, which is open on the base. 
\end{proof}

The following should be true without the assumption of properness, but we only consider that case for simplicity. 
	
\begin{prop}\label{propexpstack}
	Let $C_B\to B$ be a proper family of nodal marked curves.
	The stack of expansions $\Exp_{C_B/B}\to B$ is algebraic. It is locally of finite presentation and flat over $B$ of pure dimension zero. 
\end{prop}

The following basic lemma lets us study expansions of curves locally on the curve.

\begin{lem}\label{lemexppullback}
	Let $C_1,C_2$ be flat families of nodal marked curves over $B$ and suppose that we have an \'etale morphism $\gamma:C_1\to C_2$ over $B$ that induces isomorphisms of the singular and the marked loci. Then pullback along $\gamma$ induces an isomorphism $\Exp_{C_2/B}\to \Exp_{C_1/B}$.
\end{lem}

\begin{proof}
	The assumptions imply that the completions of $C_2$ and $C_1$ along each singular and marked locus agree. Since an expansion is an isomorphism away from these loci, the statement then follows from fpqc descent. 
\end{proof}

\begin{proof}[Proof of Proposition \ref{propexpstack}.]
Consider the stack $\mathscr{M}$ of all nodal marked curves \cite[\href{https://stacks.math.columbia.edu/tag/0DSX}{Tag 0DSX}]{stacks-project} with universal family $\mathscr{C}$, and on $\mathscr{M}\times B$, consider the morphism space $\Hom_{\mathscr{M}\times B}(\mathscr{C},C_B)$, which is an algebraic stack over $B$ locally of finite presentation \cite[\href{https://stacks.math.columbia.edu/tag/0DPN}{Tag 0DPN}]{stacks-project}. The locus that  preserves the markings is closed. The locus in which i), ii) of Definition \ref{defexpansion} holds is then open in this closed substack. Thus, we get $\Exp_{C_B/B}$ as a locally closed substack, and in particular it is algebraic and locally of finite presentation over $B$. 

For the remaining properties we may work locally on $B$, and assume without loss of generality that $B=\Spec R$ is the spectrum of a henselian local ring with separably closed residue field. Then the result follows from Lemma \ref{lemlocalexp} 
\end{proof}

\begin{lem}\label{lemlocalexp}
	Let $B = \Spec R$ be the spectrum of a henselian local ring with separably closed residue field $k$ and let $C_B\to B$ be a proper flat family of nodal marked curves with markings $\sigma_1,\ldots,\sigma_n$ and with $q_1,\ldots,q_r$ the nodes of the special fiber $C_k$. Then $\Exp_{C_B/B}\to B$ is flat of relative dimension zero, and the singularities are products of pullbacks of the singularities of the form $\AA^n\to \AA^1$ given by multiplication of the coordinates.
\end{lem}	

\begin{proof}
	By the \'etale local structure of nodes and Lemma \ref{lemexppullback} and by our assumptions on $B$, we may reduce to the case that 
	$C$ is a disjoint union of open sets of the following form:
	For each node $q_i$, there is a morphism $g_i: B\to \AA^1$, such that the component $U_i$ containing $q_i$ is isomorphic to the pullback along $g_i$ of a standard degeneration $\AA^2\to \AA^1$ given by multiplication of the coordinates. 
	Each marked point is contained in a component $V_j\simeq B\times \AA^1$, with the marking given by the origin of $\AA^1$. 
	It follows that $\Exp_{C_B/B}$ is a fiber product over $B$ of pullbacks of the stacks $\Exp_{\AA^2/\AA^1}\to \AA^1$ and $\Exp_{(\AA^1,0)}\to \Spec \CC$. These stacks can be explicitly described (see \cite[\S 1]{Li_Stab} and \cite[\S 1 + \S 6]{ACFW}): The stack $\Exp_{(\AA^1,0)}$ has a smooth cover by  affine spaces $\AA^n$, while $\Exp_{\AA^2/\AA^1}$ has a cover by affine spaces $\AA^n$, with the mapping to $\AA^1$ corresponding to multiplication of the coordinates. In particular, the morphism $\Exp_{\AA^2/\AA^1}\to \AA^1$ is flat with relative normal crossings singularities.   
\end{proof}

\subsection{Moduli spaces of vector bundles on curves}
We recall some results regarding existence of universal bundles and the structure of the Picard group for moduli spaces of stable vector bundles on curves. We also derive a version of the Bogomolov--Gieseker inequality for fiber-stable sheaves on a product with $\PP^1$.
Fix coprime integers $r,d$ with $r>0$. 
Let $F$ be a smooth projective curve of genus $g$ and let $M_F(r,d)$ denote the moduli space of rank $r$, degree $d$ stable vector bundles on $F$. For a degree $d$ line bundle $L$ on $F$, and let $M_F(r,L)$ denote the fiber of the determinant morphism $M_F(r,d) \to \Pic^d F$ over $[L]$.
\begin{prop}\label{propmodoncurve}
	Suppose that $g\geq 2$.
\begin{enumerate}[label = \roman*)]
    \item  For any $L\in \operatorname{Pic}^d F$, the Picard group of $M_F(r,L)$ is canonically isomorphic to $\mathbb{Z}$, identifying the ample generator with $1$. 
    \item Let $V$ be a vector bundle on $F$ satisfying $\operatorname{rk} V = r$ and $\deg V = r(g-1) - d$. Then for any choice of universal sheaf $\mathcal{E}^u$ on $F\times M_F(r,L)$, the line bundle
    \[L_V:= \left(\det R\pi_*(\mathcal{E}^u\otimes p^*V)\right)^{\vee}\]
    is an ample generator of $\Pic M_F(r,L)$. 
    \item \label{modonciv}There is a unique choice of universal bundle on $M_F(r,L)$ such that we have $c_1(E|_{\{pt\}\times M_F(r,d)})  \simeq L_V^{k}$ for some integer $0\leq k< r$. This $k$ is the unique integer in these bounds that satisifies $dk -r k' =1$ for some $k'\in \mathbb{Z}$. In particular, we have $(k,r)=1$. 
\end{enumerate}
\end{prop}
\begin{proof}
The first two points are Theorem B in \cite{DreNa_Group}. The last point is \cite[Remark 2.9]{Ram_Mod}.
\end{proof}

Let $L$ be a line bundle on $\mathbb{P}^1\times F$ that is the pullback of a degree $d_0\geq 1$ line bundle from $F$. We have the projection $\pi: \PP^1\times F\to \PP^1$. Let $V:= L^{\otimes r(g-1)-d}\oplus \mathcal{O}_{\PP^1\times F}^{\oplus d_0r-1}$, so that $\operatorname{rank} V =rd_0$ and $\deg V = d_0(r(g-1)-d)$. We define 
\[L_V(E) := (\det R\pi_*(E\otimes V) )^\vee.\]

A computation using Grothendieck--Riemann--Roch gives
\begin{equation}\deg L_V(E) = d_0(c_1(E)^2/2 - r\ch_2(E)) = \frac{d_0}{2}\Delta(E).
\end{equation}
We list related results:
\begin{lem}
    Let $E$ be a rank $r$ coherent sheaf satisfying $(\det E,F) = d$ and let $N$ be a line bundle on $\mathbb{P}^1$.
    Then $L_V(E) \cong L_V(E\otimes \pi^* N)$. 
\end{lem}
\begin{proof}
	One checks that $\operatorname{rk} R\pi_*(E\otimes V) =0$. The lemma then follows from the projection formula and properties of the determinant. 
\end{proof}
Let $0\leq k<r$ be the unique integer in this range satisfying $dk - rk'=1$ for some $k'$.
\begin{lem}\label{lem_mult}
  Let $E$ be a degree $r$ coherent sheaf on $\mathbb{P}^1\times F$ such that $c_1(E)$ has degree $d$ on fibers over $\PP^1$. Then   
  \begin{equation}\label{eqdegree}
   (L,c_1(E)) \equiv k\deg L_V(E) \mod d_0r.
    \end{equation}
\end{lem}
\begin{proof}
	We have that $c_1(E)= d[\PP^1\times y] +\ell [x\times F]$ for some $\ell\in \mathbb{Z}$. Thus $c_1(E)^2 = 2\ell d$, and
	\[k\deg L_V(E) =k d_0r c_2(E) - kd_0(r-1)\ell d\equiv d_0\ell k d \equiv d_0\ell = (L,c_1(E)) \mod d_0 r.\]
\end{proof}

\begin{lem}\label{lem_discpos}
	Suppose that $E$ is a rank $r$ torsion-free coherent sheaf on $\PP^1\times F$ such that the restriction of $E$ to the generic fiber over $\PP^1$ is stable of degree $d$. Then $\Delta(E)\geq 0$, with equality if and only if $E$ a tensor product of the pullbacks of a stable sheaf on $F$ and a line bundle on $\mathbb{P}^1$.
\end{lem}
\begin{proof}
	Since $E$ is torsion-free, it maps injectively to its double dual with zero-dimensional cokernel. We have $\Delta(E)\geq \Delta(E^{\vee\vee})$, with equality if and only if $E$ is locally free. By replacing $E$ with $E^{\vee \vee}$ if necessary, we may therefore assume that $E$ is locally free. 
	
	By Langton's procedure of elementary modifications \cite{Lang_Val}, one may find a locally free subsheaf $E' \subset E$ whose restriction to every fiber over $\PP^1$ is stable, and such that $E'$ is obtained from $E$ through successive elementary modifications along maximally destabilizing quotients of fibers. One checks that $\Delta(E)$ strictly decreases after each such modification, so $\Delta(E')\leq\Delta(E)$ with equality if and only if $E$ is already stable on every fiber. By replacing $E$ with $E'$ if necessary, we may assume that this is the case. 
	
	Then, after possibly tensoring $E$ by a line bundle from $\mathbb{P}^1$, we may assume that $E$ is a pull-back of the universal sheaf along a morphism $\nu:\mathbb{P}^1\to M_F(r,L)$ for some $L$. If $g=1$, this implies that $\nu$ is constant. Otherwise if $g\geq 2$, we have $\Delta(E)=2k\deg \nu^* L_V$ for $L_V$ as in \ref{propmodoncurve}. Since $L_V$ is ample, this implies that $\deg \nu^*L_V\geq 0$ with equality if and only if $\nu$ is constant. 
\end{proof}
\begin{rem}
	Without the statement about the case of equality, Lemma \ref{lem_discpos} also follows from the Bogomolov--Gieseker inequality for Gieseker-stable sheaves in view of Theorem \ref{thm_boundedwalls}.
\end{rem}
\subsection{Components of Relative Picard schemes}\label{subsec_relpic}
In this subsection, we construct a stack parametrizing connected components of the relative Picard scheme for a family of fibered surfaces. We also construct a further quotient identifying line bundles which differ by component twists. This will be used for making precise the notion of ``fixing the first Chern class'' in a family of fibered surfaces.

Let $X_B\to C_B\to B$ be a family of fibered surfaces over $B$ with structure morphism $\pi:X_B\to B$. We make the following technical assumption.
\begin{assu}\label{assu_locfree}
	The sheaf $R^1\pi_* \mathcal{O}_{X_B}$ is locally free on $B$. 
\end{assu}
\begin{rem}
This always holds if $\pi$ is representable by schemes: The family of nodal marked curves $C_B\to B$ induces natural log-structures on $B$ and $C_B$, with respect to which $C_B\to B$ is log-smooth \cite{Kato_Log}. Pulling back along $X_B\to C_B$, we get a log-structure on $X_B$, with respect to which $X_B\to B$ is log-smooth, vertical and exact. It follows from \cite[Corollary 7.1]{IKN_Quasi} that $R^1\pi_*\mathcal{O}_{X_B}$ is locally free on $B$ and commutes with any base change. 
\end{rem}
 
\begin{rem}
	Assumption \ref{assu_locfree} is likely unnecessary. One way to see this would be if one had a generalization of \cite[Corollary 7.1]{IKN_Quasi} to log-algebraic spaces. It has been pointed out to us by Luc Illusie that the proof there should work in the more general setting, see also \cite[4.2.5]{DelIll_Rel}. It would be desirable to have a direct proof in our situation that doesn't rely on logarithmic geometry. 
\end{rem}

We collect some facts about the relative Picard scheme under these hypothesis:
\begin{thm}\label{thm_picardsmooth}
	Suppose that Assumption \ref{assu_locfree} holds, for example if $X_B$ is a relative scheme over $B$.
	\begin{enumerate}[label = \roman*)]
		\item The dimension of the identity component $\Pic_{X_b}^0$ of the Picard scheme of the fiber $X_b$ is locally constant on $B$.
		\item There is an open group subscheme $\Pic^0_{X_B/B}\subset \Pic_{X_B/B}$ which restricts to the identity component of the Picard scheme over every point of $B$. 
		\item The morphism $\Pic^0_{X_B/B}\to B$ is smooth.
	\end{enumerate}
\end{thm}
\begin{proof} 
By Assumption \ref{assu_locfree}, the dimension of the tangent space of the Picard scheme $\Pic_{X_b}$ at the identity is locally constant. Since we are in characteristic zero, this implies i). Then ii) follows as in \cite[Proposition 5.20]{Klei_Pi}. Finally, iii) follows from the same argument as in \cite[Remark 5.21]{Klei_Pi} (the projectivity there is only used to invoke the GAGA theorem, which holds more generally for a proper morphism).
\end{proof}

Since the identity component is smooth over $B$, we may take the quotient  of the relative Picard scheme by it. We denote this by $\mathcal{NS}_{X_B/B}$, the \emph{relative N\'eron--Severi group}. The discussion goes through for $C_B\to B$,  $\mathcal{NS}_{C_B/B}$.

\begin{prop}
	\begin{enumerate}[label= \roman*)]
		\item The algebraic space $\mathcal{NS}_{X_B/B}$ is unramified over $B$. 
		\item The algebraic space $\mathcal{NS}_{C_B/B}$ is \'etale over $B$. 
		\item The sub-algebraic space $\mathcal{NS}_{C_B/B}^0$ parametrizing line-bundles with total degree zero is open and closed in $\mathcal{NS}_{C_B/B}$. 
		\item Pull-back of line bundles induces an open and closed immersion \[\mathcal{\mathcal{NS}_{C_B/B}}\to \mathcal{NS}_{X_B/B}.\]
	\end{enumerate}
\end{prop}
\begin{proof}
	By construction, the morphism $\mathcal{NS}_{X/B/B}\to B$ is locally of finite type and has everywhere vanishing relative Kähler differentials, which implies i). The same holds for $ \mathcal{NS}_{C_B/B}$, which is moreover smooth over $B$. Hence, it is \'etale over $B$, which gives ii).  
	Point iii) follows from the local constancy of the Euler characteristic.
	The morphism in point iv) is well defined, since the morphism $\Pic_{C_B/B}\to \Pic_{X_B/B}$ preserves the identity components. By Lemma \ref{lem_pullcl} it is a closed embedding. It follows that the induced morphism $\mathcal{NS}_{C_B/B}\to \mathcal{NS}_{X_B/B}$ is a  closed embedding. Since $\mathcal{NS}_{C_B/B} $ is \'etale over $B$ and since $\mathcal{NS}_{X_B/B}$ is unramified over $B$, it follows that $\mathcal{NS}_{C_B/B}$ is \'etale over $\mathcal{NS}_{X_B/B}$. Any \'etale (even flat) monomorphism is in particular an open immersion. 
\end{proof}

\begin{lem}\label{lem_pullcl}
	The morphism $\Pic_{C_B/B}\to \Pic_{X_B/B}$ induced by pullback is a closed immersion. 
\end{lem}
\begin{proof}	
A closed immersion is the same as a proper monomorphism.
Since $f_*\mathcal{O}_{X_B} = \mathcal{O}_{C_B}$, we have for any line bundle $L$ on $C_B$ that $f_*f^*L = L$. The the same holds after any base change on $B$, so pullback indeed gives a monomorphism. To check the existence part of the valuative criterion for properness (the uniqueness is automatic), we may assume $B = \Spec R$ for a DVR $R$ with generic point $\eta$ and that we are given a line bundle $L$ on $X_B$, so that $L_{\eta}$ is isomorphic to a line bundle pulled back from $C_{\eta}$, or equivalently so that $f_*L$ is a line bundle free on $C_{\eta}$ and that $f^*f_L\to L$ is an isomorphism over $X_{\eta}$. We want to show that these conditions hold over all of $C_R$ and $X_R$ respectively.   Since $X_R\to C_R$ is generically a family of geometrically connected curves, we can conclude that $f_*L$ is free and $f^*f_*L\to L$ an  isomorphism at least over the points where the fiber of $f$ is smooth. Thus, these conditions hold except possibly over a finite set of points $x_1,\ldots,x_n\in C_{\xi}$. But, since $X_R$ is Cohen--Macaulay, any locally free sheaf is determined by its restriction away from any codimension two locus. This implies that $f_*L$ is locally free and $f^*f_*L$ an isomorphism everywhere. 

It remains to show that $\Pic_{C_B/B}\to \Pic_{X_B/B}$ is quasi-compact. For this, we may assume that $C_B$ has constant topological type over $B$, that is, that $C_B$ is a union of smooth curves over $B$. Then one can further reduce to the case that $C_B$ is smooth over $B$. By using Chow's Lemma and possibly passing to a flattening stratification, we may assume that $X_B\to C_B$ is projective. In this case, quasi-compactness of  $\Pic_{C_B/B}\to \Pic_{X_B/B}$ follows from the stratification of the Picard scheme by Hilbert polynomials. 
\end{proof}

It follows that the quotients $\mathcal{NS}_{X_B/B}/\mathcal{NS}_{C_B/B}$ and $\mathcal{NS}_{X_B/B}/\mathcal{NS}^0_{C_B/B}$ are well-defined. We denote them by $\overline{\mathcal{NS}}_{X_B/C_B}$ and $\overline{\mathcal{NS}}_{X_B/B}$ respectively. They are separated and unramified algebraic spaces over $B$.

\begin{lem}\label{lem_pullns}
	Let $\Xtilde_T\to \Ctilde_T$ be an expansion of $X_T\to C_T$. Then pullback along $\Xtilde_T\to X_T$ induces isomorphisms $\overline{\mathcal{NS}}_{X_T/C_T} = \overline{\mathcal{NS}}_{X_B/C_B}\times_B T\xrightarrow{\sim} \overline{\mathcal{NS}}_{\Xtilde_T/\Ctilde_T}$ and $\overline{\mathcal{NS}}_{X_T/T} = \overline{\mathcal{NS}}_{X_B/B}\times_B T\xrightarrow{\sim} \overline{\mathcal{NS}}_{\Xtilde_T/T}$. 
\end{lem}
\begin{proof}
	We only treat the case of $\overline{\mathcal{NS}}$, the other is similar. Since the pullback map $\Pic_{X_T/T}\to \Pic_{\Xtilde_T/T}$ preserves the identity components, it descends to a morphism $\mathcal{NS}_{X_T/T}\to \mathcal{NS}_{\Xtilde_T/T}$. Since the total degree is preserved under pullback along $\Ctilde_T\to C_T$, this descends to a morphism $\overline{\mathcal{NS}}_{X_T/T}\to \overline{\mathcal{NS}}_{\Xtilde_T/T}$. We claim that this is an isomorphism. It is enough to show that each section of $\overline{\mathcal{NS}}_{\Xtilde_T/T}$ has a unique preimage. For this, we may work \'etale locally on $T$ and assume that $\Ctilde_T\to T$ has sections meeting each irreducible component of each fiber. Suppose that $\overline{c_1} \in \overline{\mathcal{NS}}_{\Xtilde_T/T}(S)$ for some $T$-scheme $S$. Then, locally on $S$ we may assume that $\overline{c_1}$ is represented by some line bundle $L$ on $\Xtilde_S$. Up to twisting $L$ by a line bundle $N$ pulled back from $\Ctilde_S$ of total degree zero, we may assume that $L$ is pulled back from $X_S$, so that $\overline{c_1}$ comes from an element $\overline{c_1'}\in \overline{\mathcal{NS}}_{X_T/T}(S)$. Moreover, we see that $\overline{c_1'}$ is uniquely determined, since any two possible choices of $N$ differ by a line bundle pulled back from $C_S$.  
\end{proof}

\subsection{Stability on Fibered surfaces}
We collect some results of Yoshioka that allow us to compare $f$-stability on a smooth fibered surface with slope-stability for a suitable polarization (see Theorem \ref{thm_boundedwalls})

Let $X$ be a smooth projective surface, and $f: X \to C$ be a surjective morphism to a curve $C$ with connected fibers. For a coherent sheaf $E$ on $X$ of rank $r>0$, we define its \emph{discriminant} as
\[\Delta(E) = 2 rc_2(E) -(r-1)c_1(E)^2 = c_1(E)^2-2r\ch_2(E)\in \ZZ .\]

We recall the following result of Yoshioka \cite[Lemma 2.1]{Yo_Cha}: 
\begin{lem}\label{lem_yosh1}
	Given an exact sequence $0\to G_1 \to E\to G_2 \to 0$, where $E_1$ and $E_2$ have ranks  $r_1>0$ and $r_2>0$ respectively, we have an equality
	\[\frac{1}{r} \Delta(E) = \frac{1}{r_1}\Delta(G_1)+\frac{1}{r_2}\Delta(G_2) -\frac{1}{r r_1 r_2}\left(r_2c_1(G_1)-r_1c_1(G_2)\right)^2.\]
\end{lem}
\begin{proof}By additivity of the Chern character, we have
	\begin{align*}
		& \Delta(E)/r-\Delta(G_1)/r_1-\Delta(G_2)/r_2 \\
		= &\, (c_1(G_1)+c_1(G_2))^2/(r_1+r_2)-c_1(G_1)^2/r_1 -c_1(G_2)^2/r_2 \\
		=&-\frac{1}{r r_1 r_2}\left(r_2c_1(G_1)-r_1c_1(G_2)\right)^2
	\end{align*}
\end{proof}
Let $F$ be the numerical divisor class of an arbitrary fiber of $f$ and let $H$ be an ample line bundle on $X$. For a positive rational number $t$, we let $H_t := H+t F$. 
\begin{prop}
	Let $D$ be a divisor satisfying $(D,F)\neq 0$. Suppose that $(D,H_t)=0$ for some $t>0$. Then 
	\[D^2 \leq -\frac{1}{(H,F)^2}(H^2+2t(H,F)).\]
\end{prop}
\begin{proof}
	This is \cite[Lemma 1.1]{Yo_Ell}.
\end{proof}

Now we can prove an important consequence, which already appears e.g. in \cite{Yo_Ell}, although it is stated there only for elliptic surfaces.
\begin{thm}\label{thm_boundedwalls}
	Fix values $r, \Delta \in \ZZ$ with $r>0$, and $c_1\in H^2(X,\ZZ)$, such that $r$ and $(F,c_1)$ are relatively prime. Then 
	\begin{enumerate}[label = \arabic*), left = .8\parindent]
		\item  There exists a constant $C(r,c_1,\Delta)$ so that the collection of $H_t$-semistable sheaves with rank $r$, discriminant at most $\Delta$ and first Chern class $c_1$ is independent of $H_t$ for all $t\geq C(r,c_1,\Delta)$. 
		
		\item For any $t \geq C(r,c_1,\Delta)$, semistability with respect to $H_t$ equals stability and a sheaf is stable with respect to $H_t$ if and only if its restriction to the generic fiber of $f$ is stable as a sheaf on a curve. 
	\end{enumerate}
\end{thm}
\begin{proof}
	Let $\mu_t$ denote the slope with respect to $H_t$.
	Suppose that there is a change of stability condition at $t_0$, i.e. there is a sheaf $E$ of the given invariants which is $H_t$-stable for $t=t_0$, but not for bigger (resp. smaller) values of $t$. Then we may find an exact sequence 
	\[0\to E_1 \to E\to E_2 \to 0,\]
	which is destabilizing for values of $t$ slightly bigger (resp. smaller) than $t_0$, but which consists of semistable objects for $t=t_0$ (take part of a HN filtration). In particular, $\mu_{t_0}(E_1) = \mu_{t_0}(E_2)$. By Lemma \ref{lem_yosh1} and the Bogomolov inequality applied to $E_1$, $E_2$, we have 
	\[\Delta(E)\geq -\frac{1}{r r_1r_2} (r_2c_1(E_1)-r_1c_1(E_2))^2=-D^2/(r r_1 r_2),\]
	where we set $D:=r_2c_1(E_1)-r_1c_2(E_1)$ and $r_i$ is the rank of $E_i$. The assumption that $r$ and $(c_1, F)$ are coprime implies that $(D, F)=r_2(c_1(E_1),F)-r_1(c_1(E_2),F)\neq 0$. We also have, $(D,H_{t_0})=r_1r_2(\mu_{t_0}(F_1)-\mu_{t_0}(F_2)) = 0$. Therefore, Proposition 1.2 applies to $D$, and we find that
	\[\Delta(E)\geq \frac{1}{r r_1 r_2 (H,F)^2}(H^2+2t_0 (H,F)) \geq \frac{H^2 +2t_0(H,F)}{r^3 (H,F)^2}.\]
	This shows that the values that $t_0$ can take are bounded above, so 1) follows. 
	
	To address 2), let $\eta\in C$ denote the generic point and $F_{\eta}:= f^{-1}(\eta)$. For any coherent sheaf $G$ on $X$ of rank $r_G>0$ the slope with respect to $H_t$ is
	\[\mu_t(G) =  ((c_1(G)),H)+t (c_1(G),F))/r_G.\]
	Suppose that $E$ is semistable beyond the last wall. Then for any subsheaf $E'$, we have 
	\[\mu_t(E')\leq \mu_t(E)\]
	for sufficiently large $t$. Dividing by $t$ and taking the limit as $t\to \infty$, gives  
	\[\mu(E'|_{F_{\eta}}) = (c_1(E'),F)/r'  \leq (c_1(E),F)/r = \mu(E|_F).\] 
	By the coprimeness assumption, we have strict inequality, therefore the restriction of $E$ to $F_{\eta}$ is stable. 
	
	Conversely, assume that the restriction of $E$ to the generic fiber of $f$ is stable. Then for any subsheaf $E'$, we know that $(E',F)/r'<(E,F)/r$. Let $\mu_{0,max}(E)$ be the maximum value of $\mu_0$ of a lower rank subsheaf of $E$, and let $\mu_{max}(E|_F)<\mu(E|_F)$ be the least slope of a non-zero lower rank subsheaf of $E|_F$. 
	Let $E'\subset E$ be an arbitrary lower rank subsheaf. Then we have
	\[\mu_t(E') =(E',H)/r'+t (E',F)/r'\leq \mu_{0,max}(E) + t\mu_{max}(E).  \]
	and the right hand side is strictly smaller than $\mu_t(E) = \mu_0(E) +t\mu(E|_F)$ for sufficiently large $t$, independent of $E'$. Therefore, there is no destabilizing subsheaf for sufficiently large $t$. 
	
	It follows from what we have shown that $H_t$-semistability of $E$ for $t\gg 0$ is equivalent to $H_t$-stability and equivalent to stability of $E|_{F_{\eta}}$ as desired.  
\end{proof}

\section{Moduli Spaces of sheaves on fibrations}\label{sec_modfib}
Throughout this section, we fix coprime integers $r,d$ with $r>0$. Recall that we assume all curves to be proper and connected.
\subsection{Stacks of fiber-stable sheaves}\label{subsec_modfib}

Let $f:X\to C$ be a fibered surface over a nodal marked curve $C$ as defined in Definition \ref{def_fib}. Recall that this implies that fibers of $f$ over nodes, marked points and generic points of components of $C$ are smooth projective curves.  

We define stability of sheaves relative to a fibration. 
\begin{definition}\label{deffstable}
	Let $\Ctilde \to C$ be an expansion of $C$ and let $\Xtilde := X\times_C \Ctilde$. A torsion-free coherent sheaf $E$ of rank $r$ and with fiber-degree $d$ on $\Xtilde$ is called \emph{$f$-stable}, if it satisfies the following conditions:
	\begin{enumerate}[label = \roman*)]
		\item The sheaf $E$ is locally free at the fibers of $\Xtilde \to \Ctilde$ over singular and marked points.
		\item For any generic point $\eta$ of $\Ctilde$, the restriction of $E$ to the fiber $\Xtilde_{\eta}$ over $\eta$ is slope stable.
		\item For any marked or singular point $c$ of $\Ctilde$, the restriction of $E$ to the fiber over $c$ is slope stable. 
	\end{enumerate}
\end{definition}

We extend this notion to arbitrary families. 

\begin{definition}\label{def_fstablefam}
	Let $f:X_B\to C_B$  be a family of fibered surfaces over some base $B$.
	Let $T$ be a $\mathbb{C}$-scheme. A \emph{family of $f$-stable sheaves on an expansion of $X_B$ over $B$}, valued in $T$, is given by the following pieces of data
	\begin{enumerate}[label = \roman*)]
		\item A morphism $T\to B$ with pullbacks $X_T\to C_T$,
		\item An expansion $c:\Ctilde_T\to C_T$ with associated fibered surface $\Xtilde_T:=X_T\times_{C_T}\Ctilde_T \to \Ctilde_T$.
		\item A $T$-flat coherent sheaf $E_T$ on  $\Xtilde$ such that, for every $t\in T$, the fiber $E_t$ is $f$-stable for $X_t\to C_t$. 
	\end{enumerate}
\end{definition}

Let $f:X_B\to C_B$ be a family of fibered surfaces over some base $B$.
\begin{prop}
	There is an algebraic stack $\mathcal{M}_{X_B/C_B}(r,d)$ over $B$ parametrizing $f$-stable sheaves of rank $r$ and fiber-degree $d$ on expansions of $X_B$ over $B$.
\end{prop}

\begin{rem}
We make this explicit in the case $B=\Spec\CC$.
Say $X\to C$ is a fibered surface. Then a $\CC$-point of $\mathcal{M}_{X_B/C_B}(r,d)$ is given by a pair $(c, E)$, where $c:\Ctilde\to C$ is an expansion of $C$ and $E$ an $f$-stable rank $r$ sheaf on the induced $\Xtilde:= X\times_C\Ctilde$ with fiber-degree $d$. We will sometimes write this as $(\Xtilde, E)$. 

An automorphism of the pair $(\Xtilde, E)$ is a pair $(g, \gamma)$, where $g:\Xtilde\to \Xtilde$ is an automorphism that commutes with the contraction to $X$, and $\gamma:g^*E\to E$ is an isomorphism.	
\end{rem}

\begin{proof}
	We work over the relative stack of expansions $\Exp_{C_B/B}$. Let $c:\mathcal{C}\to C_B$ denote the universal expansion, and let $\mathcal{X}:=X_B\times_{C_B}\mathcal{C} \to \mathcal{C}$ denote the induced family of fibered surfaces over $\Exp_{C_B/B}$. Let $\mathscr{M}(r,d)$ denote the stack of all torsion-free coherent sheaves of rank $r$ and fiber-degree $d$ on $\mathcal{X}$ over $\Exp_{C_B/B}$.  Then one can see that the locus of those sheaves satisfying i)-iii) of Definition \ref{deffstable} is open. 
	Thus we get $\mathcal{M}_{X_B/C_B}(r,d)\subseteq \mathscr{M}(r,d)$ as the open locus of $f$-stable sheaves.
\end{proof}

\begin{prop}
	The morphism $\mathcal{M}_{X_B/C_B}(r,d)\to B$ satisfies the existence part of the valuative criterion of properness.
\end{prop}	
\begin{proof}
	We may assume that $B = \Spec R$ for a DVR $R$ with generic point $\eta$, that $\Ctilde_{\eta}\to C_{\eta}$ is a given  expansion of $C_B$ over $\eta$ and that $E_{\eta}$ is an $f$-stable sheaf on $\Xtilde_{\eta}:=\Ctilde_{\eta}\times_C X$. Our goal is to show that we can extend this to a family of $f$-stable sheaves, possibly after passing to some extension of $R$. Let $\xi$ denote the closed point of $\Spec R$.
	
	\paragraph{Case 1: $\Ctilde_{\eta}$ is smooth.}  
	This implies that $\Ctilde_{\eta}\to C_{\eta}$ is an isomorphism. 
	Then the desired result follows from Proposition \ref{PropExtendSheaf} below.
	
	\paragraph{Case 2: $\Ctilde_{\eta}$ is of compact type.}  
	We proceed by induction on the number of components of $\Ctilde_{\eta}$. If there is only one component, we are in Case 1. Otherwise, we may decompose $\Ctilde_{\eta} = \Ctilde^1_{\eta}\cup_q \Ctilde^2_{\eta}$ along any chosen node $q$. We add a marked point $q_i$ on each $\Ctilde_i$ where the node was. Then we can find limiting families for the restrictions of $E$ to both $C^1_{\eta}$ and $C^2_{\eta}$. In order to glue the total family back together over $R$, we need to extend the isomorphism of $E^1|_{q_1}$ and $E^1|_{q_2}$ from $\eta$ over all of $\Spec R$. But since these are fiberwise isomorphic families of stable sheaves, there exists such an extensions, possibly after twisting one of $E^1$ or $E^2$ by a multiple of $\Xtilde_{\xi}^i$.  
	
	\paragraph{Case 3: $\Ctilde_{\eta}$ is arbitrary.}
	We do an induction on the first Betti number of the dual graph of $\Ctilde_{\eta}$. If it is zero, we are in Case 2. Otherwise, we may choose a non-separating node $q_{\eta}$ on $\Ctilde_{\eta}$ and take a partial normalization $\nu:\Ctilde'\to \Ctilde$ around $q_{\eta}$, while remembering the preimage of a node through adding in markings $q_{\eta,1}$ and $q_{\eta,2}$. Then by our inductive hypothesis, we can find some completion of $\Ctilde'_{\eta}$ and $\nu^*E_{\eta}$. It remains to show that we can glue together the total family along the markings $q_1$ and $q_2$. This is always possible after expanding the special fiber by, say, blowing up once along $q_1$ and then twisting by some integer multiple of the exceptional component. 
\end{proof}

\begin{prop}\label{PropExtendSheaf}
	Let $R$ be a DVR with generic point $\eta$ and let $X_R\to C_R \to \Spec R$ be a family of fibered surfaces over $\Spec R$ with $C_{\eta}$ smooth over $\eta$. Let $E_{\eta}$ be an $f$-stable sheaf of rank $r$ and fiber degree $d$ on $X_{\eta}$. 
	Then, after possibly performing a base change on $R$, we can find an expansion $\Ctilde_R\to C_R$ which is an isomorphism over $\eta$ and an $f$-stable sheaf $E_R$ on $\Xtilde_R$.
\end{prop}

\begin{proof}
	We may modify $C_R$ by repeatedly blowing up singular points to obtain a family with regular total space \cite[\href{https://stacks.math.columbia.edu/tag/0CDE}{Tag 0CDE}]{stacks-project}, which will automatically be an expansion of $C_R$. Without loss of generality, we may therefore assume that $C_R$ is nonsingular.
	
	Let $\xi$ denote the closed point of $\Spec R$.
	By the argument in \cite[Proof of the last statement of Proposition 3.3]{GiLi_Irr}, after passing to some extension of $R$ and further expanding $C_R$ over the special fiber, we can assume that $E_{\eta}$ extends to a torsion-free coherent sheaf $E_R$ on $X_R$ with the following properties:
	\begin{enumerate}[label =(\alph*)]
	\item  the restriction of $E_R$ to $X_\xi$ is torsion-free,
	\item the sheaf $E_R$ is locally free along the fibers of $X_\xi\to C_{\xi}$ over singular and marked points.  
	\end{enumerate}
	
	Let $P(C_{\xi})$ denote the collection of singular, marked and generic points of $C_{\xi}$. Unless $E_R$ is $f$-stable, there exists $x \in P(C_{\xi})$, such that the restriction $E_x$ of $E_R$ to the curve $f^{-1}(x)$ is unstable. Let $a$ be the minimum of the values $\mu_min(E_x)$, where $\mu_{min}$ denotes the minimal slope in a Harder--Narasimhan filtration. Let $\rho$ be the maximum rank of a maximally destabilizing quotient of $E_{x}$, ranging over those $x\in P(C_{\xi})$ for which $\mu_{min}(E_x) = a$. We claim that after taking a base change on $R$ and a further expansion of $C_R$, we may find a different completion of $E_{\eta}$ such that either $a$ increases or $a$ stays the same and $\rho$ decreases. Since there are only finitely many possible values for $\rho$, and since the set of possible values of $a$  lies in $\ZZ/r$ and is bounded above by $d/r$, we find that after doing so finitely many times, we end up with an $f$-stable extension of $E_{\eta}$.
	
	To prove this claim, consider the relative Quot-scheme \[q:\operatorname{Quot}_{E_R,X_R/C_R}(\rho,a)\to C_R,\] parametrizing quotients of $E_R$ of rank $\rho$ and slope $a$ on the fibers of  $X_R\to C_R$. By assumption, the image of $q$ contains some points of $P(C_{\xi})$. From the minimal choice of $a,\rho$ it follows from an analysis of the relative deformation space, that at each such point, the map $q$ is locally a closed embedding (see the proof of Theorem 5 in \cite{Ni_Sche}). 
	
	By repeatedly blowing up $C_R$ along the marked points of the special fiber, we may assume that around each marked point, the image of $q$ is supported on $C_{\xi}$. By doing a base change on $R$ and resolving the pullback of $C_R$, we may assume the same holds around any node in $C_{\xi}$. Under these assumptions, it follows that there is a closed subscheme $Z\subset C_R$, whose associated points are all in $P(C_{\xi})$ and such that around each of its associated points, it agrees with the closed subscheme defined by the Quot-scheme. 
	
	If $Z$ contains a component $C_i$ of $C_{\xi}$, we may replace $E_R$ by the elementary modification of $E_R$ along a maximally destabilizing quotient on that component. By the minimal choice of $a$ and $\rho$, the resulting sheaf will still be locally free at fibers over marked points and nodes.  This has the effect of dividing the ideal sheaf of $Z$ by the uniformizer of that component. Thus, if $Z$ is locally principal, one can do so until $Z$ becomes empty, in which case there is no more point in $P(C_{\xi})$ with maximally destabilizing subsheaf of slope $a$ and rank $\rho$. In case $Z$ is not locally principal, one can use Lemma \ref{lemmakeprinc} to find an extension $R'$ of $R$ and an expansion $c:\Ctilde_{R'}\to C_{R'}$ over $R'$ which is trivial over the generic fiber, such that the scheme-theoretic preimage of $Z$ in $\Ctilde_{R'}$ is principal. Since the relative Quot-scheme is compatible with pullback, this reduce us to the case that $Z$ is principal, which we already treated.  
\end{proof}

\begin{lem}\label{lemmakeprinc}
Let $\pi:C_R\to \Spec R$ be a nodal marked curve over a DVR $R$ with closed point $\xi$. Suppose that $\pi$ has smooth special fiber and that $C_R$ is regular. Let $Z\subset C_R$ be a closed subscheme supported on $C_{\xi}$ whose associated points are special, marked or generic points of components of the special fiber.  Then there is an extension of DVRs $R\subset R'$ and an expansion $c:\Ctilde_{R'}\to C_{R'}$ which is trivial over the generic point of $R'$ such that $c^{-1}(Z)$ has no marked points or nodes of $\Ctilde_{\xi}$ as associated points. 
\end{lem}
\begin{proof}
We first argue that by repeatedly blowing up at marked points, one can achieve that the preimage of $Z$ is principal around each marked point. Indeed, for a local calculation around the marked point $x$ we may assume that $Z$ is supported at $x$, and that the family is locally given by $\Spec R[t]\to \Spec R$, with the section given by $t=0$. Let $\pi$ be a uniformizer of $R$. Let $Z_{\pi}$ and $Z_t$ the intersections of $Z$ with the loci $(\pi = 0)$ and $(t=0)$ respectively. We claim that the invariant $\ell(Z_{\pi})+\ell(Z_{t})$ decreases for the preimage of $Z$ on the blowup. Indeed, the new marked point on the blowup is cut out by coordinates $\pi, u$, where $t= \pi u$. There exist elements $g_1 = t^a+\pi f_1$ and $g_2 = \pi^b +t f_2$ a in the defining ideal $I_Z$ of $Z$, where $a=\ell(Z_{\pi})$ and $b = \ell(Z_{t})$. Let also $k$ be maximal, so that $I_Z\subseteq (\pi,t)^k$. Thus $q^{-1}Z$ contains the exceptional divisor to order at least $k$. Note that $1\leq k\leq a,b$. Let $\Ztilde$ the non-principal part of $q^{-1 }Z$ at the marked point.  Then by a direct computation, one has 
\begin{align*}
\ell(\Ztilde_\pi) & \leq  k\\
\ell(\Ztilde_u) & \leq b - k  
\end{align*}
In particular, $\ell(\Ztilde_\pi) +\ell(\Ztilde_\pi) \leq b < \ell(Z_\pi) +\ell(Z_t)$.
Essentially the same argument works for $x$ a node in $C_{\xi}$,  where one has parameters $s,t$ locally cutting out the components of $C_{\xi}$ at $x$. Here, one needs to repeatedly blow up the nodes in the reduced preimage of $C_{\xi}$. This yields a modification $\widehat{C}_R\to C_{R}$, which principalizes $Z$ and is an isomorphism over the generic fiber, but where the fiber $\widehat{C}_\xi$ may have non-reduced components. After taking a ramified extension $R'$ of $R$ with sufficiently divisible degree, taking the normalization of $\widehat{C}_{R'}$, and resolving the singularities through repeated blow-ups, we obtain an expansion of $\Ctilde_{R'} \to C_{R'}$ with the desired properties. 
\end{proof}

\subsection{Fixing twists from the base}\label{subsec_balanced}
In the last subsection, we saw that the moduli stack $\mathcal{M}_{X_B/C_B}(r,d)$ satisfies the existence part of the valuative criterion of properness. To get a proper moduli space, we need to introduce a further numerical stability condition, which fixes twists by line bundles from $C_B$. For stability of line bundles on a curve, we use heavily ideas from \cite{OdaSe_Comp} and \cite{EP_Sem}.
  
Here, we fix $g\geq 0$ and consider only fibered surfaces $f:X\to C$ whose fibers have arithmetic genus $g$. 

For a marked nodal curve $C$ over a field, let $\operatorname{Irr}(C)$ denote the set of irreducible components of $C$. For a line bundle $N$ on $C$, we use $\deg L$ to denote the total degree, and $\undeg L$ to denote the component-wise degree, which is a function on $\operatorname{Irr}(C)$.
\begin{definition}\label{def_stability}
\begin{enumerate}[label = \roman*)]
	\item Let $C$ be a marked nodal curve over a field. A \emph{stability condition} on $C$ is a map $\alpha: \operatorname{Irr}(C)\to \mathbb{R}$. We define the \emph{degree} of $\alpha$ as $\sum_{D\in \operatorname{Irr}(C)} \alpha(D)$.
	\item Let $C_B\to B$ be a family of marked nodal curve. A $\emph{stability condition}$ on $C_B$ over $B$ is given by a collection of stability conditions $(\alpha_x)$ for each field-valued point $x$ of $B$, which are compatible in the following sense: If $\eta$ specializes to $\xi$ in $B$, there is an induced surjective morphism $\operatorname{Irr}(C_{\xi})\to \operatorname{Irr}(C_\eta)$. We require that for each $D\in \operatorname{Irr}(C_{\eta})$, we have that $\alpha_{\eta}(D)$ is equal to the sum of $\alpha_{\xi}(D')$ for $D'$ mapping to $D$. 
\end{enumerate}
\end{definition}
\begin{rem}
	It follows that for a family of curves over a finite type base $B$, a stability condition is uniquely defined by its values on the most degenerate strata. For example, if $B$ is the spectrum of a DVR, then giving a stability condition on $C_B$ is the same as giving one over the special fiber.  
\end{rem}

Let $X\to C$ be a given fibration over a nodal marked curve and let $E$ be an $f$-stable sheaf on an expansion $\Xtilde\to \Ctilde$.
\begin{definition}
	\begin{enumerate}[label = \roman*)]
	\item We say that a component of $\Ctilde$ (resp. of $\Xtilde$) is \emph{exceptional} if it is contracted by $\Ctilde\to C$ (resp. by $\Xtilde\to X$).
	
	\item	We say that the expansion $\Xtilde$ is \emph{minimal} if there is no intermediate expansion $\Xtilde\to \Xtilde' \to X$ such that $E$ is isomorphic to a pullback from $\Xtilde'$. 
		\end{enumerate}
\end{definition}

We use the following abuse of notation: Let $\Ctilde \to C$ be an expansion, and $\alpha$ a stability condition on $C$. For $D\subset \Ctilde $ an irreducible component, we set 
\[\alpha(D):=\begin{cases}
	0, & \mbox{ if $D$ is exceptional;}\\
	\alpha(c(D)) & \mbox{ if $D$ maps isomorphically to its image.} 
\end{cases}\]

For the rest of this subsection, let $f:X_B\to C_B$ a family of fibered surfaces with genus $g$ fibers over a base $B$. 
Let $L_0$ be a line bundle on $X_B$ which has positive degree $d_0$ on each fiber over $C_B$. Let $0\leq k<r$ be the unique integer, such that $kd - rk' = 1$ for some $k'\in \mathbb{Z}$. 
Define $W:= L_0^{\otimes (g-1)r-d}\oplus \mathcal{O}_X^{\oplus d_0r-1}$. For any coherent sheaf $E$ on $X$ of finite cohomological dimension, consider the line bundle 
\begin{equation}\label{eq_defW}
M(E):= \frac{\det Rf_*((\det E)\otimes L_0) }{\det Rf_*(\det E)} \otimes  \left(\det Rf_*(E\otimes W)\right)^{\otimes k}.
\end{equation}
We similarly define $M(E)$ for any expansion of $X_B$ by pulling back $L_0$.
\begin{rem}
	This definition is chosen so that $M$ has the following two properties, which is all that we will use in what follows:
	\begin{enumerate}[label = (\roman*)]
		\item (Nonzero weight) For any line bundle $N$ on $C_B$, we have 
		\[M(E\otimes f^* N) = M(E)\otimes N^{d_0 r}.\]
		\item (Normalizable on exceptional components) Let $E$ be an $f$-stable sheaf on an expansion $\Xtilde\to \Ctilde$ of $f$ over any geometric point of $B$. Then there exists a line bundle $L_1$ such that $M(E\otimes L_1)$ has degree zero on every exceptional component of $\Ctilde$. 
	\end{enumerate} 
	The first property can be seen directly from Grothendieck--Riemann--Roch. The second is a consequence of Lemma \ref{lem_mult}.
\end{rem}

We let $\alpha$ be a stability condition on $C_B$ over $B$. For a subcurve $Z\subset C$ of a nodal marked curve, we let $Z^c$ denote the ``complementary'' subcurve formed by the union of components of $C$ not contained in $Z$.

First, we consider the case where $X\to C$ is a fibered surface over a field base $B=\Spec k$.
\begin{definition}\label{def_balanced}
	Let $\Xtilde\to X$ be an expansion. 
	We say that an $f$-stable sheaf $E$ on $\Xtilde$ is \emph{$\alpha$-balanced} if 
	\begin{enumerate}[label = \arabic*)]
	\item for each proper sub-curve $\emptyset \subsetneq Z \subsetneq \Ctilde$ of $\Ctilde$ we have:
	\begin{equation}\label{eqbalanceddef}
	\left| \frac{\deg M(E)|_Z}{rd_0 } - \sum_{D\subset Z} \alpha(D) \right| \leq \frac{\#(Z\cap Z^c)}{2} ,
		\end{equation} 
	 \item For each exceptional component $D\subset \Ctilde$ we have $\deg M(E)|_D$ is non-negative. 
	
	\end{enumerate}
	We say that $E$ is \emph{strictly} $\alpha$-balanced if we moreover have 
	\begin{enumerate}[label = \arabic*)]
\setcounter{enumi}{2}
	\item If we have equality in \eqref{eqbalanceddef}, then one of $Z$ or $Z^c$ is a union of exceptional components.
		\end{enumerate}
\end{definition}
One can see that (strict) $\alpha$-balancedness is an open condition in families of $f$-stable sheaves, so this definition gives a well-behaved moduli functor for families over a general base $B$.

\begin{definition}
	We let $\mathcal{M}^{\alpha}_{X_B/C_B}(r,d)\subseteq \mathcal{M}_{X_B/C_B}(r,d)$ denote the open substack consisting of $\alpha$-balanced sheaves on minimal expansions.
\end{definition}  

\begin{prop}\label{prop_valex}
	The stack $\mathcal{M}^{\alpha}_{X_B/C_B}(r,d)$ satisfies the existence part of the valuative criterion of properness.
\end{prop}
\begin{proof}
	Let $R$ be a DVR with generic and closed points $\eta$ and $\xi$, and with a given morphism $\Spec R\to B$. Let 
	  $E_{\eta}$ be a sheaf on an expansion $\Xtilde_{\eta}\to \Ctilde_{\eta}$ of $X_{\eta}\to C_{\eta}$, such that $E_{\eta}$ is $f$-stable and strictly $\alpha$-balanced. By Proposition \ref{PropExtendSheaf}, we can find \emph{some} extension $\Xtilde_R\to \Ctilde_R$ of this data to an $f$-stable family $E_R$ (possibly after replacing $R$ by an extension).  We claim that we can modify this data to obtain an $\alpha$-balanced bundle on a minimal expansion.
	
	For this, pick a line bundle $L_0$ on $\Ctilde_R$ that has degree zero on exceptional components of $\Ctilde_{\xi}$ and such that $ M(E_R)\otimes L_0$ has degree a multiple of $rd_0$ on each component of $C_{\xi}$ (this may require further extending $R$). Then by Lemma \ref{lem_mult} we can pick $L_1$ on $\Ctilde_R$, such that $\underline{\deg}\, L_1|_{\Ctilde_\xi}^{\otimes rd_0} = \underline{\deg} \,(M(E|_{Z})\otimes L_0)|_{\Ctilde_{\xi}}$. By changing $L_0$, we may in fact assume without loss of generality that $L_1^{\otimes r d_0} \cong M(E) \otimes L_0$ on $\Ctilde_R$.
	
	Now consider the stability condition $\alpha' := \alpha + \undeg \, L_0 / (r d_0)$. Then we have the following lemma, whose proof is straightforward:
	\begin{lem}\label{lemstabequivi}
	 $E$ is $\alpha$-balanced if and only if $L_1$ is $\alpha'$-semistable in the following sense:
	 $L_1$ has non-negative degree on exceptional components, and for every subcurve $\emptyset\subsetneq Z \subsetneq \Ctilde$, we have
	
	\[\left| \deg L_1|_Z - \sum_{D\subset Z} \alpha'(D) \right| \leq \frac{\# (Z\cap Z^c)}{2}.\]
	\end{lem}

	We consider the coherent sheaf $L_1':=c_*L_1$ on $C_R$ with adjunction map $\psi:c^*L_1'\to L_1$. If this is surjective, we obtain an induced map $P:\Ctilde_{R}\to \mathbb{P}(L_1')$. 
	\begin{lem}\label{lemstabequivii}
	The following are equivalent (over the generic and closed point of $R$ respectively):
	\begin{enumerate}[label = \roman*)]
	\item The line bundle $L_1$ is $\alpha'$-semistable in the sense of Lemma \ref{lemstabequivi}.
	\item 	\begin{enumerate}[label = (\alph*)]
		\item The sheaf $L_1'$ is torsion-free, the morphism $\psi$ is surjective and identifies $L_1$ with the pullback along $P$ of the universal quotient of $L_1'$, and
		\item $L_1'$ is Oda-Seshadri $\alpha'$-semistable in the sense of \cite[Definition 4.1]{KP_Stab}.
	\end{enumerate}
	\end{enumerate}
 
	\end{lem}
	\begin{proof}
	By the arguments in \cite[\S 5]{EP_Sem}, it follows that ii), (a) is equivalent to the condition that $L_1$ has only degrees $0,1$ on exceptional components, and total degree at most $1$ on each chain of exceptional components. 
	
	Then, one can check by hand that $\alpha'$-stability for $L_1$ (in the sense of Lemma \ref{lemstabequivi}) and Oda-Seshadri $\alpha'$-semistability for $L_1'$ are equivalent by using a destabilizing subcurve for the one to construct one for the other.
	\end{proof}
	
	As in \cite[Corollary 4.3]{KP_Stab}, it follows from Simpson stability, that any $\alpha'$-semistable torsion-free sheaf on the generic fiber has an $\alpha'$-semistable limit.
	Let $L'_2$ be an $\alpha'$-semistable limit that agrees with $L_1'$ on the generic fiber. After possibly further expanding $\Ctilde_{\xi}$, we can assume that we have a morphism $P_2:\Ctilde_R\to \mathbb{P}(L_2')$, and denote by $L_2$ the pullback of the universal line bundle along $P_2$. Note that this implies that $L_1$ and $L_2$ are isomorphic over $\eta$. We consider the line bundle $L_E:=L_2\otimes L_1^{\vee}$ on $\Ctilde_R$.
	
	\paragraph{Claim:}$E_1:= E\otimes L_E$ is $\alpha$-balanced. 
	
	To see this, note that $M(E_1) = M(E)\otimes L_E^r$.
	Therefore  
	\[\undeg \, M(E_1) = \undeg\, M(E) +r(\undeg \,L_2 -\undeg \,L_1) = -\undeg L_0 +r\undeg L_2.\]
	In particular, using the reverse direction of Lemma \ref{lemstabequivii} and Lemma \ref{lemstabequivi}, we find that $L_2$ is $\alpha'$-stable, and that $E_1$ is $\alpha$-balanced. Since $L_E$ is trivial along $\Xtilde_{\eta}$, we find that $E_1$ is an $\alpha$-balanced extension of $E_{\eta}$. 
	
	Finally, one can obtain a minimal expansion $\Ctilde$ by contracting the components $D$ of $\Ctilde_{\xi}$ over which $E_1$ is isomorphic to a pullback along $F\times D\to F$, where $F$ is the fiber of $X_{\xi}\to C_{\xi}$ over the image of $D$. Since we assumed that $\Xtilde_{\eta}\to \Ctilde_{\eta}$ was minimal, one can do this contraction without affecting the generic point. This uses that every component of $\Xtilde_{\eta}$ contains in its closure at least one component that is not contracted, which one can see for example using Lemma \ref{lem_discpos}.	
\end{proof}

\begin{prop}\label{prop_finaut}
	Let $E$ be a strictly $\alpha$-balanced $f$-stable sheaf on an expansion $\Xtilde\to \Ctilde$. Then the subgroup of scalar automorphisms has finite index in the automorphism group of $(E,\Xtilde)$.
\end{prop}
\begin{proof}
	  By $f$-stability, every automorphism of $E$ as a sheaf on $\Xtilde$ must be scalar: Since the restriction to each fiber over a generic point $\eta$ of $\Ctilde$ is geometrically stable, any automorphism must be scalar over a dense open of $\Ctilde$, and therefore scalar since $\Ctilde$ and $E$ is flat over $\Ctilde$. 
	  
	In particular, for every automorphism $\gamma$ of $\Ctilde$ over $C$, there exists at most one isomorphism $\phi:\gamma^* E\to E$ up to scaling. On the other hand, each automorphism of $\Ctilde$ is given by scaling exceptional components. Let $D\subset \Ctilde$ be an exceptional component. By restricting to $D$, we may assume that without loss of generality, we may assume that $D=\mathbb{P}^1$, and that $\gamma$ acts by multiplication with $a\in \mathbb{G}_m$. Then the restriction $E_D$ of $E$ to $X_D=\mathbb{P}^1\times F$ is stable on the generic fiber over $\mathbb{P}^1$ (and over the fibers over $0,\infty$). If the map $\nu_E:\mathbb{P}^1\to M_F(r,d)$ induced by $E_D$ is nontrivial, then it is finite onto its image, and $a$ must preserve the fibers of $\nu_E$. In particular, there are only finitely many possible values for $a$. If $\nu_E$ is constant, there might still be distinguished points in $\mathbb{P}^1$ over which the restriction of $E_D$ to the fiber is not locally free or not stable, in this case again $a$ must permute the finite set of those points, so must belong to a finite set. If neither of these occur, then $E_D$ is a pullback of a stable bundle from $F$ twisted by the pullback of a degree $\ell$ line bundle from $\mathbb{P}^1$. By $\alpha$-balancedness, we have $\ell$ is $0$ or $1$, and by minimality of $\Xtilde$ we must have $\ell = 1$. We claim that on any such component, $a$ must be an $rd_0$'th root of unity, which shows that there are only finitely many possible choices of $\gamma$.
	
	We now show this last claim. 
	Since scaling fixes the points $0, \infty\in \mathbb{P}^1$, we have that $\phi$ induces an automorphism of the restriction of $E$ to the fibers over $0,\infty$, say $\phi_0,\phi_\infty$, which are given by scalar multiplication. They are related by $\phi_{\infty} = a^{-1}\phi_0$. 
	In particular, if $a$ is not an $rd_0$'th root of unity, then $\phi$ induces an automorphism of the pair $(\Ctilde, M(E))$ that is given by scaling $M(E)$ differently at different nodes which are fixed by $\gamma$.  Without loss of generality, we may assume that there is one node of $\Ctilde$ at which this scaling is trivial. Then let $Z\subseteq \Ctilde$ be the maximal connected subcurve of $\Ctilde$ containing this node, such that at all nodes in $Z$, the automorphism induced by $\phi$ is trivial. Then any irreducible component $D'$ of $\Ctilde$ intersecting $Z$ in a finite set must be exceptional, and have $\deg M(E)|_{D' } = rd_0$. If $Z$ is a chain of exceptional components, this means inequality \eqref{eqbalanceddef} is violated. Otherwise, for this $Z$, the \emph{strict} inequality in \eqref{eqbalanceddef} is violated. In either case, this contradicts the assumption that $E$ is strictly $\alpha$-balanced. 
\end{proof}

\subsection{Boundedness and Properness}\label{subsec_proper}
Let $f:X_B\to C_B$ be a family of fibered surfaces over $B$. We assume here that $B$ is connected. Let $L_0$ be a line bundle on $X_B$ with degree $d_0>0$ on fibers over $C_B$ and let $\alpha$ be a stability condition on $C_B$. 
To get a bounded moduli space, we need to fix numerical invariants. For this, we will use the relative N\'eron--Severi scheme of a family constructed in \S \ref{subsec_relpic}. In order for the results there to apply, we will impose Assumption \ref{assu_locfree} from here until the end of \S \ref{sec_degen}.

We fix a section $\overline{c_1}\in \overline{\mathcal{NS}}_{X_B/B}(B)$ which has fiberwise degree $d$, and $\Delta\in \ZZ$.
\begin{definition}
	We let $\mathcal{M}_{X_B/C_B}^{\alpha} (r,\overline{c_1},\Delta)\subseteq \mathcal{M}^{\alpha}_{X_B/C_B}(r,d)$ denote the substack of sheaves whose fiberwise discriminant is $\Delta$ and for which the associated section of $\overline{\mathcal{NS}}_{X_B/B}$ defined by the determinant agrees with the pullback of $\overline{c_1}$.
	
	If $C_B$ has geometrically irreducible fibers over $B$, then $\mathcal{NS}_{X_B/B}= \overline{\mathcal{NS}}_{X_B/B}$, and we also use the notation $\mathcal{M}_{X_B/C_B}(r,c_1,\Delta)$ for $c_1\in \mathcal{NS}_{X_B/B}$.
\end{definition}
Note that this makes sense by Lemma \ref{lem_pullns}.
\begin{lem}
	The stack $\mathcal{M}_{X_B/C_B}^{\alpha}(r,c_1,\Delta)$ is an open and closed substack of $\mathcal{M}_{X_B/C_B}^{\alpha}(r,d)$
\end{lem}
\begin{proof}
 The condition that the degree of the cycle $2 r c_2(E) -(r-1)c_1(E)^2$ equals $\Delta$ is an open and closed condition. Since $\overline{\mathcal{NS}}_{X_B/B}$ is separated and unramified over $B$, the section $c_1$ determines an open and closed subspace.
\end{proof} 
 
\begin{prop}
	The stack $\mathcal{M}^{\alpha}_{X_B/C_B}(r,\overline{c_1},\Delta)$ is of finite type over $B$.
\end{prop}

\begin{proof}
	It only remains to show that it is quasi-compact over $B$. We may work locally on $B$ and stratify $B$ by the singularity type of $C_B$. In particular, we may assume that $B$ is a finite type $\CC$-scheme and that the singular locus of $C_B$ is a disjoint union of copies of $B$.
	Let $C_1,\ldots,C_n$ denote the components of $C_B$, which are smooth over $B$, and $X_1,\ldots,X_n$ their preimages under $f$. 
	\paragraph{Claim 1:} There exists an integer $N_1$, such that for every $b\in B$, every $\alpha$-balanced $f$-stable sheaf $E$ on an expansion of $X_b$ and every $i$, we have $\Delta(E|_{X_i})\geq N_1$.
	\begin{proof}
		Since the discriminant is invariant under tensoring with a line bundle, we may assume that $c_1(E|_{X_i}) = c_1|_{X_i} + k F$ for a fixed lift $c_1$ of $\overline{c_1}$ and some $k\in \{0,\ldots r-1\}$. By Theorem \ref{thm_boundedwalls}, we can find a polarization $H_i$ on $X_i$ such that $f$-stability for $X_i\to C_i$ agrees with slope stability with respect to $H_i$ for all sheaves of rank $r$, first Chern class of the form $c_1|_{X_i} + k F$ and with discriminant at most $0$, say. In particular, the collection of such sheaves is bounded, so their discriminant is bounded below by some constant $N_1^i$. Taking the minimum of all the $N_1^i$ gives us the desired $N_1$.    
	\end{proof}
	\paragraph{Claim 2:} There exists a number $N_2$, so that for an $\alpha$-balanced $f$-stable sheaf on a minimal expansion $\Xtilde_b\to \Ctilde_b$, the number of exceptional components is at most $N_2$.
	\begin{proof}
		By Lemma \ref{lem_discpos}, on each exceptional component $Y$ of $\Xtilde_b$,  $E_Y$ is either a pullback tensored by a line-bundle from $\Ctilde$, or $\Delta(E_Y)>0$. By $\alpha$-balancedness, there can be at most $g(C_b)$ components for which the first possibility occurs (and the line bundle has to be of degree one on the corresponding component). It follows that the total number of exceptional components is bounded by $\Delta - Irr(C_b)N_1 + g(C_b)$.
	\end{proof}
	From these two claims, it follows that there is an a-priori bound for $\Delta(E|_Y)$ for any $f$-stable sheaf on an expansion, and $Y$ an arbitrary component of the expansion.  
	
	From Claim 2, we also see that $\mathcal{M}^{\alpha}_{X_B/C_B}(r,\overline{c_1},\Delta)$ factors through a quasi-compact open subset of $\Exp_{C_B/B}$. Thus we may further reduce to the case that the preimage of each stratum is in $\Exp_{C_B/B}$ is quasi-compact. This reduces us to showing that the space of sheaves on a given expansion $\Ctilde_B\to C_B$ is quasi-compact. By $\alpha$-stability, for each component there is only a finite choice of possibly values that the first Chern class of a restriction can take. 
	
	In this case, an $f$-stable sheaf on $\Xtilde_B$ is the same as giving suitable $f$-stable sheaves on each component, together with isomorphisms. This reduces us to the case where $\Xtilde_B$ has a single component. In this case, the space of $f$-stable sheaves with given $\overline{c_1}$ and $\Delta$ is open in the space of stable sheaves with respect to a suitably chosen polarization, hence quasi-compact.  
\end{proof}

\begin{definition}
	Let $X_B\to C_B\to B$ be a family of fibered surfaces. 
	We say that a stability condition $\alpha$ on $C_B$ is generic, if every $\alpha$-balanced sheaf on a minimal expansion of $X_B$ is in fact strictly $\alpha$-balanced. 
\end{definition}

\begin{rem}
	Suppose that $C_B\to B$ has a single most degenerate stratum over a closed point $b$. Then one can always choose a non-degenerate stability condition, by choosing an $\alpha:\operatorname{Irr}(C_b)\to \mathbb{R}$, whose values form a $\mathbb{Q}$-vector space of dimension $|\operatorname{Irr}(C_b)|-1$.
\end{rem}

\paragraph{Properness.}
 
\begin{definition}
	Let $\alpha$ be a generic stability condition on $C_B$. 
 	We denote the $\mathbb{G}_m$-rigidification of $\mathcal{M}_{X_B/C_B}^{\alpha}(r,\overline{c_1},\Delta)$ along the scalar automorphisms by
 	\[M_{X_B/C_B}^{\alpha}(r,\overline{c_1},\Delta).\] 
\end{definition}

By Proposition \ref{prop_finaut}, for a choice of generic stability condition, the stack $M_{X_B/C_B}^{\alpha}(r,\overline{c_1},\Delta)$ has finite stabilizer groups at every point, and therefore is Deligne-Mumford. 

\begin{thm}\label{thm_properdm}
	Let $\alpha$ be a generic stability condition. Then $M_{X_B/C_B}^{\alpha}(r,\overline{c_1},\Delta)$ is a proper Deligne--Mumford stack over $B$.
\end{thm}
\begin{proof}
We already know that it is a finite type Deligne--Mumford stack. It satisfies the existence part of the valuative criterion of properness, since by Proposition \ref{prop_valex} this is true for $\mathcal{M}^{\alpha}_{X_B/C_B}(r, d)$. It only remains to address the uniqueness part of the valuative criterion. 
For this, we may assume that $B=\Spec R$, and that we are given expansions $\Xtilde_1\to \Ctilde_1$ and $\Xtilde_2\to \Ctilde_2$ of $X_B\to C_B$ together with respective $\alpha$-balanced $f$-stable sheaves $E_1$ and $E_2$ and an isomorphism $\Psi$ of the restrictions to the generic fibers. Let $\pi$ be a uniformizer for $R$. Then we need to show that $\Ctilde_1\simeq \Ctilde_2$ and that for some $\ell$, the isomorphism $\pi^\ell \Psi$ can be extended to an isomorphism of $E_1$ and $E_2$ over the isomorphism of expansions. 

We first choose a common further expansion $\Ctilde_1\xleftarrow{c^1}\Ctilde_3\xrightarrow{c^2} \Ctilde_2$ which is an isomorphism over $\eta$ and is minimal in the sense that no component of $\Ctilde_3$ over $\xi$ is contracted by both $c^1$ and $c^2$. 
Then, both $(c^1)^*E_1$ and $(c^2)^*E_2$ are $\alpha$-balanced $f$-stable sheaves on $\Xtilde_3$ and we have a given isomorphism $\psi: (c^1)^*E_1|_{\Xtilde_{3,\eta}} \to (c^2)^*E_2|_{\Xtilde_{3,\eta}}$. Let $\pi$ be a uniformizer of $R$. There is a unique choice of integer $\ell$, such that $\pi^{\ell}\psi$ extends to a morphism $E_1\to E_2$ whose restriction to $\Xtilde_{3,\xi}$ is nonzero. This extension is then unique, and we will denote it again by $\pi^{\ell}\Psi$. By Lemma \ref{lem_exun} below, $\pi^\ell \Psi$ is an isomorphism. In particular, any component of $\Ctilde_{3}$ which is contracted by $c^1$ is also contracted by $c^2$, so we have $\Ctilde_1\cong \Ctilde_3\cong \Ctilde_2$. This is precisely what we wanted to show.
\end{proof}

\begin{lem}\label{lem_exun}
	Let $f:X_R\to C_R$ be a family of fibered surface over a DVR, let $\alpha$ be a stability condition on $C_R$, and let $E_1$ and $E_2$ be strictly $\alpha$-balanced $f$-stable sheaves on some expansion $\Xtilde_R\to \Ctilde_R$ of $f$. Let $\Psi: E_1\to E_2$ be a morphism that is an isomorphism over the generic point of $R$ and non-zero over the closed point. Suppose that $\Delta(E_1) = \Delta(E_2)$ and that $c_1(E_1) \equiv c_1(E_2)$ in $\mathcal{NS}_{X_B/B}$. 
	Then $\Psi$ is an isomorphism. 
\end{lem}
\begin{proof}
After possibly taking a base change on $R$ and a further expansion of $\Ctilde_R$, we may assume without loss of generality that all irreducible components of $\Ctilde_R$ are regular.
Since $E_1$ and $E_2$ are flat over $\Ctilde_R$, the sheaf $L_0:= f_*\mathcal{H}om(E_1,E_2)$ is a reflexive rank one sheaf on $\Ctilde_R$ and isomorphic to the structure sheaf over $\Ctilde_{\eta}$. 	
As a reflexive sheaf, it is locally free away from the singular points of $\Ctilde_R$, and since it is locally free on the generic fiber, these are the finitely many points $x_i$ in the special fiber in which two irreducible components of $\Ctilde_R$ intersect. By restricting to an open $U_i$ only intersecting of a given $x_i$ the two adjacent components, there is a unique integer $\ell_i\geq 0$, such that $\pi^{- \ell_i}\Phi$ is well-defined on $U$ and is non-zero on the special fiber on one of the components. It follows that it is an isomorphism on the fiber over $x_i$, and hence that it is a generator of $f_*\mathcal{H}om(E_1,E_2)$ around $x_i$. This implies that $L_0$ is indeed locally free. We get a tautological morphism $E_1\otimes L_0\to E_2$, which is non-zero on each component of the special fiber.  Therefore its restriction to $\Xtilde_{\xi}$ is injective with cokernel $Q$ supported on fibers of $\Xtilde_{\xi}\to \Ctilde_{\xi}$. Since $E_1$ and $E_2$ have the same numerical invariants, and $L_0$ has total degree zero on $\Ctilde_{\xi}$, we find that $Q=0$, so $E_1\otimes L_0\simeq E_2$. This implies that both $E_1$ and $E_1\otimes L_0$ are strictly $\alpha$-balanced. This implies that $L_0$ must have degree zero on each non-exceptional component of $\Ctilde_{\xi}$ (otherwise, such a component gives a subcurve violating stability). Since $L_0$ is trivial on $\Ctilde_{\eta}$, this is enough to conclude that in fact $L_0\simeq \mathcal{O}_{\Ctilde}$. Since $\Phi$ is non-zero on $\Xtilde_{\xi}$, it gives a generator of $L_0$, so it must be an isomorphism by what we already argued. 
\end{proof}

\begin{rem}\label{rem_alphas}
Note that by \eqref{eq_defW} and Grothendieck--Riemann--Roch, the total degree of $M(E)$ depends on $E$ only through $c_1(E)$ and $\Delta(E)$. In particular, a formal application of Grothendieck--Riemann--Roch gives a unique number $\alpha(\overline{c_1},\Delta)$  such that $M_{X/C}^\alpha(r,\overline{c_1}, \Delta) = \emptyset$ unless 

\begin{equation}\label{eq_nonemptycondition}
\sum_{D\in \operatorname{Irr}(C)} \alpha(D) = \alpha(\overline{c_1},\Delta).
\end{equation}   

When $C$ is irreducible, a stability condition is just a scalar which determines whether the moduli space is (possibly) nonempty. In this case, we will abbreviate 
\[M^b_{X/C}(r,\overline{c_1},\Delta) := M_{X/C}^{\alpha(\overline{c_1},\Delta)}(r,\overline{c_1},\Delta).\]
\end{rem}

\subsection{Perfect Obstruction Theories}\label{subsec_perf}
We construct the perfect obstruction theory on the moduli stacks $\mathcal{M}_{X_B/C_B}(r,d)$ and their variants. The arguments here are relatively standard and we will not give all details. A \emph{perfect obstruction theory} for a morphism $\mathcal{X}\to \mathcal{Y}$ is an object $E\in D(\mathcal{X})$ that is perfect with amplitude in $[-1,1]$ together with a morphism $E\to L_{\mathcal{X}/\mathcal{Y}}$ that is an isomorphism on $h^1$ and $h^0$ and surjective on $h^{-1}$. This coincides with the usual notion whenever $\mathcal{X}\to \mathcal{Y}$ is of DM-type.  

Let $X_B\to C_B$ be a family of fibered surfaces. We abbreviate $\mathcal{M}:= \mathcal{M}_{X_B/C_B}(r,d)$ and $\Exp:=\Exp_{C_B/B}$. Consider the forgetful morphisms $\mathcal{M}\to \Exp\to B$. We let $\Xtilde\to \Ctilde$ denote the universal expansion on $\Exp$ and let $\mathcal{E}$ denote the universal sheaf on the pullback $\Xtilde_{\mathcal{M}}$ of $\Xtilde$ to $\mathcal{M}$. Let $\pi:\Xtilde_{\mathcal{M}} \to \mathcal{M}$ denote the projection. 
The Atiyah class defines a relative obstruction theory $(R\pi_*R\mathcal{H}om_0(\CE,\CE))^{\vee}[-1]\to L_{\mathcal{M}/\Exp}$. 

We have a factorization $\mathcal{M}\to \mathcal{P}ic_{\Xtilde/\Exp}\to \Exp$ of the forgetful map through the determinant morphism to the Picard stack. Let $\mathcal{L}$ denote the universal line bundle over $\Xtilde_{\mathcal{P}ic_{\Xtilde/B}}$ and let $\pi$ also denote the projection to $\mathcal{P}ic_{\Xtilde/B}$. We have the relative obstruction theory $(R\pi_*R\mathcal{H}om_0(\mathcal{L},\mathcal{L}))^{\vee}[-1]\to L_{\mathcal{P}ic_{\Xtilde/Exp}/Exp}$. It is naturally compatible with the obstruction theory of $\mathcal{M}$ via the trace map. 

Moreover, the trace-free part gives a canonical relative obstruction theory $(R\pi_*R\mathcal{H}om_0(\CE,\CE))^{\vee}[-1]\to L_{\CM/\mathcal{P}ic_{\Xtilde/\Exp}}$. We have a commutative diagram involving the $\mathbb{G}_m$-rigidifications of both stacks
\[\begin{tikzcd}
	\mathcal{M} \ar[r,"r"] \ar[d]& M:=M_{X_B/C_B}(r,d) \ar[d]\\
	\mathcal{P}ic_{\Xtilde/\Exp} \ar[r]& Pic_{\Xtilde/Exp}.
\end{tikzcd}
  \]
 The induced map $r^* L_{M/Pic_{\Xtilde/\Exp}}\to L_{\mathcal{M}/\mathcal{P}ic_{\Xtilde/\Exp}}$ is an isomorphism, so we may compose with its inverse to get a morphism $(R\pi_*R\mathcal{H}om_0(\CE,\CE))^{\vee}[-1]\to r^* L_{M/Pic_{\Xtilde/Exp}}$. 
 This last map descends to a canonical perfect obstruction theory for the determinant morphism $M\to \Pic_{\Xtilde/\Exp}$ over the relative Picard scheme.
 From this discussion, and the properties of virtual pullback, we immediately get
 \begin{prop}\label{prop_perfob}
 	Suppose we are in the situation of Theorem \ref{thm_properdm}.  
 	Then the stack $M^{\alpha}_{X_B/C_B}(r,\overline{c_1},\Delta)$ has a relative perfect obstruction theory over $\Pic_{\Xtilde/\Exp}$. In particular, it has a natural virtual fundamental class given by virtual pullback of the fundamental class of $\Pic_{\Xtilde/\Exp}$. 
 	The formation of the virtual fundamental class is compatible with flat and l.c.i. pullbacks on $B$. 
 \end{prop}
 
 \subsection{Evaluation maps}\label{subsec_eval}
  Let $X_B\to C_B$ be a family of fibered surfaces and let $\sigma_1,\ldots,\sigma_n : B\to C_B$ denote the markings of $C_B$. Let $F_i\to B$ denote the family of smooth curves obtained as the preimage of $\sigma_i$ under $f$. For each $i$, we have a morphism of stacks $\mathcal{M}_{X_B/C_B}(r,d)\to \mathcal{M}_{F_i}(r,d)$. It fits into a commutative diagram
 \begin{equation*}
 	\begin{tikzcd}
 		\mathcal{M}_{X_B/C_B}(r,d)\ar[r]\ar[d]& \mathcal{M}_{F}(r,d)\ar[d] \\
 		\mathcal{P}ic_{\Xtilde/B}\ar[r] & \mathcal{P}ic_{F_i/B}
 	\end{tikzcd}
 \end{equation*}
in which the horizontal maps are the restriction maps and the vertical maps are the determinant morphisms. For each map in this square, the obstruction theories for source and target are naturally compatible. We have an induced square of rigidifications, with induced relative obstruction theories
\begin{equation*}
	\begin{tikzcd}
		M_{X_B/C_B}(r,d)\ar[r]\ar[d]&M_F(r,d) \ar[d] \\
	\Pic_{\Xtilde/B}\ar[r] & \Pic_{F_i/B}.
	\end{tikzcd}
\end{equation*} 
\subsection{Tautological classes}
We want to study invariants which are defined by pairing certain tautological cohomology classes against the virtual fundamental class. We define here what we mean by tautological cohomology class. 
For a Deligne-Mumford stack $\mathcal{Y}$ over $\CC$, we define its rational (co-) homology groups $H_*(\mathcal{Y}, \QQ)$ (resp. $H^*(\mathcal{Y}, \QQ)$) in terms of the simplicial scheme $Y_\bullet$ associated to an \'etale cover $Y_0\to \mathcal{Y}$. When working with rational coefficients -- as we do here -- these are naturally isomorphic to the (co-) homology groups of the coarse moduli space of $Y$. This is nicely explained in the second half of \cite{Beh_Coh}. One also has a natural cycle class map $A_*(\mathcal{Y})\to H^{BM}_*(\mathcal{Y}, \QQ)$ into the Borel--Moore homology (cf. \cite[\S 2]{AGV_Grom}). When $\mathcal{Y}$ is proper, this equivalently gives a map $A_*(\mathcal{Y})\to H_*(\mathcal{Y},\QQ)$.

Let $f:X\to C$ be a fibered surface over a fixed nodal marked curve $C$ with markings $(x_1,\ldots,x_n)$ and let $F_1,\ldots,F_n$ denote the fibers over the markings. Let $L_0$ be a fixed line bundle of degree $d_0>0$ on $X$ and let $\overline{c_1}\in \overline{\mathcal{NS}}_{X}$ be a class of fiber degree $d$. Let also $\Delta\in \mathbb{Z}$. 
Let $\alpha$ be a generic stability condition on $C$. We consider the proper Deligne-Mumford stack $M:= M^{\alpha}_{X/C}(r,\overline{c_1}, \Delta)$. Let $\pi:\Xtilde\to M$ denote the universal expansion over $M$ and $c:\Xtilde\to X$ the contraction map. 
\begin{lem}
	There is a natural map $\pi_{!}: H^*(\Xtilde,\mathbb{Q})\to H^{*-4}(M,\mathbb{Q})$. 
\end{lem}
\begin{proof}
	Since the morphism $\pi:\Xtilde\to M$ is flat, proper and representable of dimension $2$, any \'etale cover $M_0\to M$ induces an \'etale cover $\Xtilde_0 \to \Xtilde$ by pullback, and we get an induced morphism of simplicial algebraic spaces $\pi_{\bullet}:\Xtilde_{\bullet}\to M_{\bullet}$, which is component-wise flat and proper of relative dimension two. we have a trace map $(R\pi_{\bullet})_*\QQ \to \QQ [-4]$ (see \cite[4.6]{Ver_RR} for the case of schemes, which carries over to our setting). In fact, fiberwise, $R^4\pi_*\QQ$ is a $\QQ$ vector space spanned by the orientation classes of irreducibel components in the fiber, and the map $R^4\pi_*\QQ\to \QQ$ sends each generator to $1$. This induces the desired morphism after passing to cohomology groups.
\end{proof}

Let $\gamma\in H^*(X)$ be a cohomology class. Recall that $M$ is the rigidification of the moduli stack $\mathcal{M}_{X/C}^{\alpha}(r,\overline{c_1}, \Delta)$ and similarly $\Xtilde$ is the rigidification of a family $\widetilde{\mathcal{X}}$. We denote by $\mathcal{E}$ the universal sheaf on $\widetilde{\mathcal{X}}$. While $\mathcal{E}$ does not descend to a $\Xtilde$, the expression $\mathcal{E}\otimes (\det \mathcal{E})^{-(1/r)}$ makes sense as a $K$-theory class and does descend to $\Xtilde$. By abuse of notation, we denote it by $\widehat{\mathcal{E}}$. We make the following definition for $i\geq 0$: 
\begin{equation}\label{eq_deftaut}
\ch_i(\gamma):= \pi_{!}\left(\ch_i(\widehat{\mathcal{E}}) \cup c^*\gamma\right).
\end{equation}
By abuse of notation, we denote by $T_M^{\vir} = -[R\Hom_0(\mathcal{E},\mathcal{E})]$ the $K$-theory class dual to the relative perfect obstruction theory of $M$ over the relative Picard scheme. 
Here is an (incomplete) definition of tautological classes.
\begin{definition}
	We say that a cohomology class in $H^*(M,\QQ)$ is \emph{tautological}, if it lies in the sub-ring generated by classes $\ch_i(\gamma)$ and classes $\ch_i(T_M^{\vir})$.
\end{definition}

More generally, one can also consider classes defined in terms of $K$-theoretic objects, such as virtual Segre or Verlinde invariants (see for example \cite{GK_Sheaves} for an overview). 
 
\section{The degeneration formula} \label{sec_degen}
In this section we state and prove a special case of a degeneration formula, when the base curve has one node and two irreducible pieces. 

Let $f:X\to C$ be a fibered surface and suppose $C=D_1\cup D_2$, where $(D_1,x_1)$ and $(D_2,x_2)$ are smooth curves with a single marking and the union is taken along the marked points. Let $Y_i:=f^{-1}D_i$ and let $F_i:= f^{-1}(x_i)$ for $i=1,2$. 
 
We fix some $L_0$ with fiber degree $d_0>0$ on $X$ and a stability condition $\alpha$ on $C$.  For applications one may always choose $L_0$ as $L$ or $L^{-1}$.
 
 We assume that $\alpha(D_i)\not\in \frac{1}{r d_0}\mathbb{Q}$, in particular that the stability condition $\alpha$ is generic. We let $\alpha_i:= \alpha(D_i)$. 
 We further fix a section $\overline{c_1}$ of $\overline{\mathcal{NS}}_{X}$. We will use the following abuse of notation: If $c_1'$ and $c_1''$ are points in $\mathcal{NS}_{Y_1}$ and $\mathcal{NS}_{Y_2}$ respectively, we write $c_1'+c_1''= \overline{c_1}$ if there exists a lift $c_1$ of $\overline{c_1}$ to $\mathcal{NS}_X$ which restricts to $c_1'$ and $c_1''$ on $Y_1$ and $Y_2$ respectively. 
 
 Finally, we write 
 \[\mathcal{M}_{Y_1/D_1}^{\rough{\alpha_1}}(r,c_1',\Delta_1):=\coprod_{\substack{\beta\in \frac{1}{rd_0}\QQ\\|\beta-\alpha_i|<\frac{1}{2}}}\mathcal{M}_{Y_1/D_1}^{\beta}(r,c_1',\Delta_1),\]
 and similarly for $\mathcal{M}_{Y_2/D_2}^{\rough{\alpha_2}}(r,c_1'',\Delta_2)$ and the $\mathbb{G}_m$-rigidified versions of the stacks. 
Note that at most one term in the disjoint union is nonempty. 
\subsection{Glueing of sheaves}\label{subsec_glueing}
Let $c_1'\in \mathcal{NS}_{Y_1}$, $c_1'' \in \mathcal{NS}_{Y_2}$ such that $c_1'+c_2'' = \overline{c_1}$. Suppose that $\alpha$ satisfies \eqref{eq_nonemptycondition}. Let $\Delta_1,\Delta_2 \in \ZZ$ and $\Delta:=\Delta_1+\Delta_2$. 
There is an associated glueing morphism 
\[\gamma: \mathcal{M}_{Y_1/D_1}^{\rough{\alpha_1}}(r,c_1',\Delta_1)\times_{\mathcal{M}_{F}(r,d)}\mathcal{M}_{Y_2/D_2}^{\rough{\alpha_2}}(r,c_1'', \Delta_2)\to \mathcal{M}^{\alpha}_{X/C}(r, \overline{c_1}, \Delta_1 + \Delta_2).\]

This induces a canonical morphism on $\mathbb{G}_m$-rigidifications
\[\Gamma: M_{Y_1/D_1}^{\rough{\alpha_1}}(r,c_1',\Delta_1)\times_{M_F(r,d)} M_{Y_2/D_2}^{\rough{\alpha_2}}(r,c_1'',\Delta_2)\to M^{\alpha}_{X/C}(r,\overline{c_1}, \Delta).\] 
 
We similarly have a glueing morphism for Picard schemes
\[\Pic_{\Ytilde_1/\Exp_{D_1,x_1}}^{c_1'}\times_{\Pic_F^d} \Pic^{c_1''}_{\Ytilde_2/\Exp_{D_2,x_2}}\to \Pic_{\Xtilde/\Exp_C}^{\overline{c_1}}.\]
Here we use a superscript to denote the component of the Picard schemes mapping into the respective component of the N\'eron--Severi schemes. 
The glueing morphisms are compatible with taking the determinant, i.e. we have a commutative diagram 
\begin{equation}
	\begin{tikzcd}
	 M_{Y_1/D_1}^{\rough{\alpha_1}}(r,c_1',\Delta_1)\times_{M_F(r,d)} M_{Y_2/D_2}^{\rough{\alpha_2}}(r,c_1'',\Delta_2)	\ar[r,"\Gamma"]\ar[d]& M^{\alpha}_{X/C}(r,\overline{c_1}, \Delta). \ar[d] \\
	\Pic^{c'_1}_{\Ytilde_1/\Exp_{D_1}}\times_{\Pic_F^d} \Pic^{c_1''}_{\Ytilde_2/\Exp_{D_2}} \ar[r] & \Pic^{\overline{c_1}}_{\Xtilde/\Exp_C} 
	\end{tikzcd}
\end{equation} 

By taking the union over possible decompositions of the discriminant, we can obtain more
\begin{lem}\label{lem_decfirst}
	The following natural diagram is cartesian: 
	\begin{equation}\label{eq_detglue}
		\begin{tikzcd}
		\coprod\limits_{\Delta_1+\Delta_2 = \Delta}	M_{Y_1/D_1}^{\rough{\alpha_1}}(r,c_1',\Delta_1)\times_{M_F(r,d)} M_{Y_2/D_2}^{\rough{\alpha_2}}(r,c_1'',\Delta_2)	\ar[r,"\Gamma"]\ar[d]& M^{\alpha}_{X/C}(r,\overline{c_1}, \Delta) \ar[d] \\
			\Pic^{c'_1}_{\Ytilde_1/\Exp_{D_1}}\times_{\Pic_F^d} \Pic^{c_1''}_{\Ytilde_2/\Exp_{D_2}} \ar[r,"\Gamma'"] & \Pic^{\overline{c_1}}_{\Xtilde/\Exp_C} 
		\end{tikzcd}
	\end{equation} 
\end{lem}
\begin{proof}
This can be checked before passing to rigidifications. In that case, the fiber product corresponding to the lower right corner of the diagram has $T$-points given by an element of $\mathcal{M}^{\alpha}_{X/C}(r,\overline{c_1},\Delta)(T)$ -- i.e. a sheaf $E_T$ on an expansion $\Xtilde_T$ -- together with a choice of decomposition $\Xtilde_T=\Ytilde_{T,1}\cup \Ytilde_{T,2}$ of the given expansion, such that we have $[\det E|_{\Ytilde_{T,1}}] =  c_1'$ and $[\det E|_{\Ytilde_{T,2}}] = c_1''$. The stability condition on $C$ implies that the restrictions will lie in the prescribed range of stabilities on the $D_i$. 
\end{proof}
It follows from Lemma \ref{lem_decfirst}, that we have a natural relative obstruction theory on each product $M_{Y_1/D_1}^{\rough{\alpha_1}}(r,c_1',\Delta_1)\times_{M_F(r,d)}M_{Y_2/D_2}^{\rough{\alpha_2}}(r,c_1'',\Delta_2)$ over the product of relative Picard schemes $\Pic_{\Ytilde_1/\Exp_{D_1}}^{c_1'}\times_{\Pic_F}\Pic_{\Ytilde_2/\Exp_{D_2}}^{c_1''}$, given by pulling back the obstruction theory via $\Gamma$. In particular, we have a canonical virtual fundamental class on each product, obtained by pulling back the fundamental class of the base. 
\begin{prop}\label{prop_virdec}
	We have an equality in $A_*(M_{X/C}^{\alpha}(r,\overline{c_1}, \Delta))$.
	\[[M_{X/C}^{\alpha}(r,\overline{c_1}, \Delta)]^{\vir} =\hspace{-7pt} \sum_{\substack{c_1'+ c_1'' = \overline{c_1}\\ \Delta_1 + \Delta_2 = \Delta}}\Gamma_*[M_{Y_1/D_1}^{\rough{\alpha_1}}(r,c_1',\Delta_1)\times_{M_F(r,d)}M_{Y_2/D_2}^{\rough{\alpha_2}}(r,c_1'',\Delta_2)]^{\vir}.\]
\end{prop}
\begin{proof}
By \cite[Theorem 4.1]{Mano_Pull} and Lemma \ref{lem_decfirst}, it is enough to show the analogous formula on the level of Picard schemes, i.e. that
\[ [\Pic_{\Xtilde/\Exp_C}] = \sum_{c_1'+ c_1'' = \overline{c_1}} \Gamma'_*[\Pic_{\Ytilde_1/\Exp_{D_1}}^{c_1'} \times_{\Pic_F^d} \Pic_{\Ytilde_2/\Exp_{D_2}}^{c_1''} ] \]
 and that the glueing morphisms of relative Picard schemes are projective. This is shown in Lemma \ref{lem_gluePic} below.
\end{proof}
 
\begin{lem}\label{lem_gluePic}
	\begin{enumerate}[label = \roman*)] 
		\item The glueing map 
		\[\Gamma': \Pic_{\Ytilde_1/\Exp_{D_1}}^{c_1'}\times_{\Pic_F}\Pic_{\Ytilde_2/\Exp_{D_2}}^{c_1''}\to \Pic_{\Xtilde/\Exp_C}^{\overline{c_1}}\]
		is a proper and quasi-finite morphism, in particular it is projective. 
		\item We have an equality of fundamental classes
		\[[\Pic_{\Xtilde/\Exp_C}^{\overline{c_1}}] = \sum_{c_1'+c_2'' = \overline{c_1}} \Gamma_*([\Pic_{\Ytilde_1/\Exp_{D_1}}^{c_1'}\times_{\Pic_F}\Pic_{\Ytilde_2/\Exp_{D_2}}^{c_1''}]).\]
		This equality is to be understood whenever one restricts both sides to a quasi-compact open subset $U$ of $\Pic_{\Xtilde}^{\overline{c}_1}$ and its respective preimages $\Gamma^{-1}U$ under the glueing maps. 
	\end{enumerate}
\end{lem}
\begin{proof}
	Note that $\Pic_{\Ytilde_1/\Exp_{D_1}}^{c_1'}\times_{\Pic_F}\Pic_{\Ytilde_2/\Exp_{D_2}}^{c_1''}$ is smooth over $\Exp_{D_1}\times \Exp_{D_2}$. Indeed, it is a union of connected components of $\Pic_{\Ytilde_1\cup_F\Ytilde_2/\Exp_{D_1}\times \Exp_{D_2}}$, which are all translates of the identity component. 
	
	Regarding i): One shows that $\Gamma'$ is quasi-compact by an argument similar to the one used in the proof of Lemma \ref{lem_pullns}. Then, it is straightforward to check the valuative criteria for properness.  Quasi-finiteness, follows since for a given point $L_1$ of $\Pic_{\Xtilde}^{\overline{c_1}}$ defined on an expansion $\Xtilde_1\to \Ctilde_1$, points in the preimage under $\Gamma$ correspond to a choice of singular point in $\Ctilde_1$. 
	Now ii) follows, since each glueing map is between schemes of the same dimension. Thus, to compute the image of the fundamental cycle $[\Pic_{\Ytilde_1/\Exp_{D_1}}^{c_1'}\times_{\Pic_F}\Pic_{\Ytilde_2/\Exp_{D_2}}^{c_1''}]$ it is enough to do so over a dense open on target and source. Hence, we may restrict to the generic points of $\Exp_{D_i}$ and $\Exp_C$ corresponding to trivial expansions. Then the result follows, since we have an isomorphism 
	\[\Pic_{X}^{\overline{c_1}} = \coprod_{c_1'+c_1'' = \overline{c_1}}\Pic_{Y_1}^{c_1'}\times_{\Pic_F}\Pic_{Y_2}^{c_1''}.\] 
	
\end{proof}  
\subsection{Obstruction theories}
Let $\Delta_1,\Delta_2\in \ZZ$ with $\Delta = \Delta_1+\Delta_2$. For convenience, we abbreviate 
\begin{align*}
	M_{Y_1/D_1} &:= M_{Y_1/D_1}^{\rough{\alpha_1}}(r,c_1',\Delta_1), \\
	\Pic_{\tilde{Y}_1} &:= \Pic^{c_1'}_{\Ytilde_1/\Exp_{D_1}},\\
	M_F &:= M_F(r,d),
\end{align*}
and similarly for $M_{Y_2/D_2}, \Pic_{\Ytilde_2}, M_{X/C}^{\alpha}$ and $\Pic_{\Xtilde}$.
In this subsection, we further analyze the virtual class on $M_{Y_1/D_1}\times_{M_{F}} M_{Y_2/D_2}$ that was constructed in \S \ref{subsec_glueing}.

\begin{prop}\label{prop_virclass}
	\begin{enumerate}[label = \arabic*)]
	\item There is a relative perfect obstruction theory for the morphism $M_{Y_i/D_i}\to M_F$ (for $i=1,2$), which induces a canonical virtual pullback map. 
	\item The following cycle classes on $M_{Y_1/D_1}\times _{M_F} M_{Y_2/D_2}$ agree: 
	\begin{enumerate}[label = \roman*), left = .3 \parindent]
		\item The virtual pullback of the fundamental class of $\Pic_{\Ytilde_1}\times_{\Pic_F^d}\Pic_{\Ytilde_2}$,
		\item The Gysin-pullback of the product of virtual classes on $M_{Y_1/D_1}\times M_{Y_2/D_2}$ along the diagonal map $M_F\to M_F\times M_F$. 
		\item the virtual pullback of the fundamental class of $M_F$ induced by the morphism $M_{Y_1/D_1}\to M_F$,
		\item the virtual pullback of the fundamental class of $M_F$ induced by the morphism $M_{Y_2/D_2}\to M_F$.
	\end{enumerate}	
	\end{enumerate}
\end{prop}
\begin{proof}

	We have a natural commutative diagram
	\begin{equation*}
		\begin{tikzcd}
			M_{Y_1/D_1}\ar[r,"\varphi"]\ar[dr]& \Pic_{\Ytilde_1}\times_{\Pic_F^d}M_F\ar[r]\ar[d]& M_F\ar[d] \\
			&\Pic_{\Ytilde_1} \ar[r] & \Pic_F^d
		\end{tikzcd}
	\end{equation*}
in which the square is cartesian. The vertical maps have natural obstruction theories given by the trace-free part of the Atiyah class of a universal sheaf over $M_F$, and these are naturally compatible with the obstruction theory of $M_{Y_1/C_1}$ over $\Pic_{\Ytilde_1}$. It follows that we have a (non-canonical) relative perfect obstruction theory for the morphism $\varphi:M_{Y_1/D_1}\to \Pic_1\times_{\Pic_{F}^d}M_F$. Since the horizontal maps in the square are l.c.i., we may endow them with their canonical obstruction theory, which are then automatically compatible with the obstruction theory for $\varphi$. There is then an induced obstruction theory for the restriction map $M_{Y_1/D_1}\to M_F$, which has the property that the induced virtual pullback map factors through l.c.i pullback along $\Pic_{\Ytilde_1}\to \Pic_F^d$ followed by the virtual pullback along $\varphi$. Note that since $\Pic_{\Ytilde_1}$ is not Deligne--Mumford (or quasi-compact), there may be subtleties as to how l.c.i. pullback along $\Pic_{\Ytilde_1}\to \Pic^d_F$ interacts with virtual pullbacks. Since one can exhaust $\Pic_{\Ytilde_1}$ by global quotients of algebraic spaces, one can work $GL$-equivariantly on a suitable principal bundle. This proves the first point for $i=1$, and by symmetry for $i=2$.

Now, consider the commutative diagram with cartesian squares
\begin{equation*}
	\begin{tikzcd}
		M_{Y_1/D_1}\times_{M_F}M_{Y_2/D_2}\ar[r]\ar[dr]&M_{Y_1/D_1}\times_{\Pic_F^d} M_{Y_2/D_2}\ar[r]\ar[d]& M_{Y_1/D_1}\times M_{Y_2/D_2}\ar[d] \\
		&\Pic_{\Ytilde_1}\times_{\Pic_F^d}\Pic_{\Ytilde_2}\ar[r]\ar[d] & \Pic_{\Ytilde_1}\times \Pic_{\Ytilde_2} \ar[d]\\
		&\Pic_{F}^d\ar[r,"\Delta"]& \Pic^d_F\times \Pic^d_F
	\end{tikzcd}
\end{equation*}
The vertical maps in the upper square have obstruction theories given by the sum of Atiyah classes. The obstruction theory of the left vertical map is compatible with the one of the diagonal map given by the Atiyah class. We get an induced obstruction theory on the map $M_{Y_1/D_1}\times_{M_F} M_{Y_2/D_2}\to M_{Y_1/D_1}\times_{\Pic_F^d}M_{Y_2/D_2}$ which is isomorphic to $\Ext_0^1(\mathcal{E}_D,\mathcal{E}_D)^{\vee}$ concentrated in degree $-1$. On the other hand, we have the cartesian diagram
\begin{equation*}
	\begin{tikzcd}
		M_{Y_1/D_1}\times_{M_F}M_{Y_2/D_2}\ar[r]\ar[d]& M_{Y_1/D_1}\times_{\Pic_F^d} M_{Y_2/D_2}\ar[d] \\
		M_F\ar[r,"\Delta"] & M_F\times_{\Pic_F^d}M_F
	\end{tikzcd}
\end{equation*}
Since virtual pullback is independent of the precise choice of map in the obstruction theory, this shows that the virtual pullback map for the morphism $M_{Y_1/D_1}\times_{M_F}M_{Y_2/D_2}\to M_{Y_1/D_1}\times_{\Pic_F^d}M_{Y_2/D_2}$ is equal to the Gysin-pullback along the diagonal of $M_F$. Then, considering the diagram with cartesian squares
\begin{equation*}
	\begin{tikzcd}
		M_{Y_1/D_1}\times_{M_F}M_{Y_2/D_2}\ar[r]\ar[d]&M_{Y_1/D_1}\times_{\Pic_F^d}M_{Y_2/D_2}\ar[r]\ar[d]& M_{Y_1/D_1}\times M_{Y_2/D_2}\ar[d] \\
		M_F\ar[r,"\Delta"]&M_F\times_{\Pic_F^d}M_F\ar[r]\ar[d] & M_F\times M_F\ar[d]\\
		&\Pic_F^d\ar[r,"\Delta"]&\Pic_F^d\times \Pic_F^d
	\end{tikzcd}
\end{equation*}
shows that the virtual class on $M_{Y_1/D_1}\times_{M_F}M_{Y_2/D_2}$ is the Gysin-pullback of the one on $M_{Y_1/D_1}\times M_{Y_2/D_2}$ along the diagonal morphism of $M_F\times M_F$. 
The last two equivalences follow, since virtual pullback commutes with l.c.i. pullback, and the fact that $M_{Y_1/D_1}\times_{M_F} M_{Y_2/D_2}$ is identified with the base change of $(M_{Y_1/D_1}\times M_F) \times_{(M_F\times M_F)} (M_F\times M_{Y_2/D_2})$ along the diagonal $M_F\to M_F\times M_F$.\end{proof}

\subsection{Decomposition formulas}

We show some basic results regarding how tautological classes interact with the glueing morphism. 
Let $\gamma$ be a cohomology class on $X$, let $\gamma_i$ be its restriction to $Y_i$ for $i=1,2$, and let $\gamma_F$ be its restriction to the singular fiber $F$. Consider the glueing map 
\[\Gamma: M_{Y_1/D_1}\times_{M_F(r,d)} M_{Y_2/D_2}\to M_{X/C}\]

\begin{lem}\label{lem_split1}
	We have $\Gamma^*\ch_i(\gamma) = \operatorname{pr}_1^*\ch_i(\gamma_1) + \operatorname{pr}_2^*\ch_i(\gamma_2)$.
\end{lem}
\begin{proof}
Let $\Xtilde\to M_{X/C}$ denote the universal expansion, and $\Gamma^*\Xtilde$ its pullback to $M_{Y_1/D_1}\times_{M_F(r,d)} M_{Y_2/D_2}$, so that $\Gamma^*\Xtilde = \operatorname{pr}^*_1\Ytilde_1\cup_F\operatorname{pr}_2^*\Ytilde_2$. Recall that $\ch_i(\gamma)$ is defined via \eqref{eq_deftaut} in terms of a Gysin map $\pi_!$, which commutes with the pullback along $\Gamma$. Then consider the following diagram of maps
\begin{equation*}
	\begin{tikzcd}
		\operatorname{pr}_1^*\Ytilde_1\coprod \operatorname{pr}_2^*\Ytilde_2 \ar[r,"\sigma"]\ar[dr,"\pi_{12}"]& \operatorname{pr}_1^*\Ytilde_1\cup_F \operatorname{pr}_2^*\Ytilde_2 \ar[d,"\pi_{\Gamma}"] \\
		& M_{Y_1/D_1}\times_{M_F(r,d)} M_{Y_2/D_2}
	\end{tikzcd}
\end{equation*}
Then one can check that we have an identity $(\pi_\Gamma)_! = (\pi_{12})_!\circ \sigma^*$ of maps $H^*(\operatorname{pr}_1^*\Ytilde_1\cup_F \operatorname{pr}_2^*\Ytilde_2)\to H^{*-2}(M_{Y_1/D_1}\times_{M_F(r,d)} M_{Y_2/D_2})$. The lemma follows from this.
\end{proof}

We consider the restriction of classes derived from the virtual tangent bundle. 
\begin{lem}\label{lem_split2}
	We have 
	\[\Gamma^*\ch_i(T_{M_{X/C}}^{\vir}) = \operatorname{pr}_1^*\ch_i(T_{M_{Y_1/D_1}}^{\vir})+\operatorname{pr}_2^*\ch_i(T_{M_{Y_2/D_2}}^{\vir}) - \operatorname{pr}_F^*\ch_i(T_{M_F/\Pic_F}).\]
\end{lem}
\begin{proof}
The obstruction theory on $M_{X/C}$ is a descent of $R\Hom_0(\mathcal{E},\mathcal{E})$, where $\mathcal{E}$ is the universal sheaf on the un-rigidified moduli stack.  Since $[\Gamma^*\mathcal{E}] =[\operatorname{pr}_1^*\mathcal{E}_1] + [\operatorname{pr}_2^*\mathcal{E}_2] - [\operatorname{pr}_F^*\mathcal{E}_F]$, it follows from adjunction that 
\[T_{M_{X/C}}^{\vir} = \operatorname{pr}_1^*T_{M_{Y_1/D_1}}^{\vir} + \operatorname{pr}_2^*T_{M_{Y_2/D_2}}^{\vir} - \operatorname{pr}_F^*T_{M_{Y_F}/\Pic_F},\]
and all of these are perfect objects. The result follows from this. 
\end{proof}

\subsection{Fixed determinant spaces} \label{subsec_fixdet}
In order to give more precise statements for some of the invariants we consider, we want to work in some `fixed determinant' theory. We make this precise here in the two cases we are interested in: For a simple degeneration and for a surface fibered over a smooth curve with a single marked point. 

\paragraph{Simple Degeneration.}
Let $B$ be regular one-dimensional base, and $X_B\to C_B$ a family of fibered surface. Assume the total space $C_B$ is regular, that $C_B\to B$ is smooth outside $b_0\in B$, and that $C_{b_0}$ is a union of two components along a simple node. We say that $X_B\to C_B$ is a simple degeneration of fibered surfaces.

We consider the stack $\Exp_{C_B/B}\to B$ with universal expansions $\Xtilde_B\to \Ctilde_B$, and the relative Picard scheme $\Pic_{\Xtilde_B/\Exp_{C_B/B}}$. 

For the following lemma, we introduce some notation: Given an \'etale morphism $\beta: B\to \mathbb{A}^1$, such that $b_0 = \beta^{-1}(0)$, let $B[n]:=B\times_{\AA^1}\AA^{n+1}$ and $C_B[n]\to B[n]$ be the standard degeneration as in \cite[\S 1.1]{Li_Stab}. Let also $X_B[n]:=X_B\times_{C_B} C_B[n]$. This defines a smooth morphism $\beta_n :B[n] \to \Exp_{C_B/B}$. Say $Y_1$ and $Y_2$ are the irreducible components of $X_{b_0}$. Then let $Y_{1,k}\subset X_B[n]$ denote the divisor that corresponds to $Y_1$ over the $k$-th coordinate hyperplane. 

\begin{lem}\label{lem_detfix}
	\begin{enumerate}[label = \alph*)]
		\item There is a minimal closed sub-stack $\overline{e}\subset \Pic_{\Xtilde_{B}/B}$ through which the identity section $\Exp_{C_B/B}\to \Pic_{\Xtilde_B/\Exp_{C_B/B}}$ factors.
		\item The stack $\overline{e}$ is naturally a subgroup of $\Pic_{\Xtilde_{B}/\Exp{C_B/B}}$ and the structure map $\overline{e}\to \Exp_{C_B/B}$ is \'etale.
		\item For $\beta:B\to \AA^1$ as above, we have that $\beta[n]^{-1}\overline{e}\subseteq \Pic_{X_B[n]/B[n]}$ is equal to the reduced subscheme supported on the union of sections defined by line bundles $\mathcal{O}_{X_B[n]}(\sum_{i=1}^{n+1} a_k Y_{1,k})$ for $a_k\in \ZZ$, which is a closed set.
	\end{enumerate}
\end{lem}
\begin{proof}
	We will show that the union of sections $\mathcal{O}_{X_B[n]}(\sum_{i=1}^{n+1} a_k Y_{1,k})$ defines a closed sub-space of $\Pic(X_B[n]/B[n])$, which is the closure of the identity section $B[n]\to \Pic_{X_B[n]/B[n]}$. It is then straightforward to see that the collection of such closed substacks for all $n$ induces a closed substack of $\Pic_{\Xtilde_B/B}$, which is the minimal substack containing the identity section.
	
	Since the pull-back $\Pic_{C_B[n]/B[n]}\to \Pic_{X_B[n]/B[n]}$ is closed, it is enough to show the analogous statement for $\Pic_{C_B[n]/B[n]}$. Let $D_{1,k}$ denote the image of $Y_{1,k}$ in $C_{B}[n]$. 
	Since the identity component $\Pic^0_{C_B[n]/B[n]}$ is separated, it follows that the identity section is closed in it and doesn't contain any other section that agrees with the identity section generically over $B[n]$. Then for any other section $\xi$ given by a line bundle $\mathcal{O}_{X_B[n]}(\sum a_k D_{1,k})$, we get a closed immersion $\xi \subseteq  \xi \Pic^0_{C_B[n]/B[n]}$. It follows that the union over all such sections $\xi$ is a closed subset in $\cup_{\xi} \xi \Pic^0_{X_B[n]/B[n]}$. But in fact $\cup_{\xi} \Pic^0_{C_B[n]/B[n]}= \Pic_{C_B[n]/B}$.  
	
	To see b), it is enough to show that $\cup_{\xi} \xi\subset \Pic_{C_B[n]/B[n]}$  is a subgroup-space and \'etale over $B[n]$. The first point is clear, since the collection of $\xi$'s forms a group. To see that it is \'etale, we may work locally on the domain. But $\cup_{\xi} \xi\cap \xi_0 \Pic^0_{C_B[n]/B[n]} = \xi_0$, which is clearly \'etale over $B[n]$. 
\end{proof}

Let $L$ be a line bundle on $X_B$ with degree $d$ on fibers over $C_B$, let $\alpha$ be a generic stability condition on $C_B$ and let $L_0$ be a line bundle on $X_B$ with fiber degree $d_0>0$.  Let $\Delta\in \mathbb{Z}$.
\begin{definition}
	We let $\mathcal{M}^{\alpha}_{X_B/C_B}(r, L, \Delta)$ denote the moduli stack of $\alpha$-balanced $f$-stable sheaves on minimal expansions of $X_B$ whose determinant map factors through $L\overline{e}$ and whose discriminant is $\Delta$. 
	 This is naturally a closed substack of $\mathcal{M}^{\alpha}_{X_B/C_B}(r, \overline{c_1(L)}, \Delta)$, where $\overline{c_1(L)}$ is the section $\Exp_{C_B/B}\to \overline{\mathcal{NS}}_{X_B/B}$ induced by $c_1(L)$.
	We also let $M^{\alpha}_{X_B/C_B}(r, L, \Delta)$ denote the $\mathbb{G}_m$-rigidification of $\mathcal{M}^{\alpha}_{X_B/C_B}(r, L, \Delta)$. 
\end{definition}

\begin{rem}
We have an analogous result if $X\to C$ is a fibered surface with $C$ a union of two smooth curves along a single node (e.g. if $X\to C$ is the central fiber of a simple degeneration, but without assuming a smoothing exists). We use the notation $\mathcal{M}_{X/B}^{\alpha}(r,L,\Delta)$ and $M_{X/B}^{\alpha}(r,L,\Delta)$ for the resulting moduli stacks. We leave the details to the reader, who may alternatively always assume that we are working with the central fiber of a simple degeneration.
\end{rem}

\paragraph{Expansions.}
Let $(C,x)$ be a smooth marked curve and $X\to C$ be a fibered surface. We consider the stack of expansions $Exp_{C}$ and the relative Picard-scheme $\Pic_{\Xtilde/\Exp_{C}}$.   The stack $\Exp_C$ has a cover by affine spaces $\alpha_n:\mathbb{A}^n\to \Exp_{C}$, together with a standard expansion $X[n]\to C[n]$ \cite[\S 4.1]{Li_Stab}. We let $Y_{k}$ denote the closure of the component induced by $X$ over the $k$-th coordinate hyperplane in $\mathbb{A}^n$. 
Then, we have the analogue to Lemma \ref{lem_detfix}, with essentially the same proof.
\begin{lem}
	\begin{enumerate}[label = \alph*)]
		\item There is a minimal closed sub-stack $\overline{e}\subseteq \Pic_{\Xtilde/\Exp_C}$ through which the identity section $\Exp_{C,x}\to \Pic_{\Xtilde_B/\Exp_{C,x}}$ factors.
		\item The stack $\overline{e}$ is naturally a subgroup of $\Pic_{\Xtilde/\Exp_{C}}$ and the structure map $\overline{e}\to \Exp_{C,x}$ is \'etale.
		\item Given an \'etale morphism $\beta: B\to \mathbb{A}^1$,  Then, $\alpha_n^{-1}\overline{e}\subseteq \Pic(X[n]/\AA^n)$ is equal to the reduced subscheme supported on the union of sections defined by line bundles $\mathcal{O}_{X_B[n]}(\sum_{i=1}^{n+1} a_k Y_{k})$ for $a_k\in \ZZ$, which is a closed set.
	\end{enumerate}
\end{lem}
\begin{proof}
\end{proof}
Let $L$ be a line bundle on $X$. Let $\alpha\in \mathbb{Q}$ be arbitrary.
\begin{definition}
	We let $\mathcal{M}_{X/C}^{\alpha}(r,L,\Delta)$ denote the moduli stack of $\alpha$-balanced $f$ -stable sheaves on minimal expansions of $X$ whose determinant map factors through $L\overline{e}$ and whose discriminant is $\Delta$. This is naturally a closed substack of $\mathcal{M}^{\alpha}(r,\overline{c_1}(L), \Delta)$. We also denote by $M^{\alpha}_{X/C}(r,L,\Delta)$  the $\mathbb{G}_m$-rigidification. 
\end{definition}

The results of this section carry over to this setting in the obvious way. In particular this applies to Lemmas \ref{lem_split1} and \ref{lem_split2} and to Propositions \ref{prop_virdec} and \ref{prop_virclass}.

For clarity, we restate Proposition \ref{prop_virdec} for this setting explicitly. For this, suppose we are in the situation of Proposition \ref{prop_virdec} and that we have also fixed a line bundle $L$ on $X$ in class $\overline{c}_1$. By abuse of notation, write $L_1+L_2 = \overline{L}$, if $L_1$ and $L_2$ are line bundles on $Y_1$ and $Y_2$ respectively, whose restrictions to $F$ are isomorphic and such that there exists an integer $\ell$ with $L|_{Y_1}\simeq L_1(-\ell F)$ and $L|_{Y_2}\simeq L_2(\ell F)$. Let also $L_F:=L|_F$. 

\begin{prop}\label{prop_fixdetfundclass}
	We have an equality in $A_*(M^{\alpha}_{X/C}(r,L,\Delta))$. 
	\[[M_{X/C}^{\alpha}]^{\vir} = \sum_{\substack{L_1 + L_2 = \overline{L}\\ \Delta_1+\Delta_2 = \Delta}} \Gamma_*[M_{Y_1/D_1}^{\rough{\alpha_1}}(r,L_1,\Delta_1)\times_{M_F(r,L_F)} M_{Y_2/D_2}^{\rough{\alpha_2}}(r,L_2,\Delta_2)]^{\vir}\]
\end{prop} 
\subsection{Invariants}\label{subsec_invs}
We define the type of invariants that we want to consider in the degeneration formula. For simplicity, we restrict the discussion to a specific type of invariant and to the fixed determinant case. We expect that similar formulas hold for, say Segre and (with some more work) Verlinde invariants. It also shouldn't be essential to work with the fixed determinant version, but then one should consider insertions coming from the Picard scheme in order to get non-trivial invariants when the two theories differ.

Let $A$ be a multiplicative genus (e.g. the Chern polynomial, $\chi_y$-genus or elliptic genus. Let $X\to C$ be a fibered surface, where $C$ is either smooth, with possibly a marked point, or a union of two irreducible components along a single node. We assume we have fixed data $\alpha, L, L_0$ as in \S \ref{subsec_fixdet}.  

For $\mathcal{M}\to M$ a moduli stack and its $\mathbb{G}_m$-rigidification as considered throughout, we will consider cohomology classes of the form 
\begin{equation}\label{eq:insertions}
\Phi(\mathcal{E}) = A(T^{\vir})B(\mathcal{E}).
\end{equation}
Here
\begin{enumerate}[label = \roman*)]
\item $\mathcal{E}$ denotes the universal sheaf on some family of expansions over $\mathcal{M}$ of a given fibered surface $X\to C$,
\item $A$ is a multiplicative transformation from the $K$-theory of perfect objects to cohomology with coefficients in some ring $\Lambda$ containing $\QQ$,
\item $B(\mathcal{E}) = \exp(\sum_{i,\gamma} \ch_i(\gamma) q_{\gamma,i})$, where $i$ ranges through integers $\geq 2$ and $\gamma$ ranges through a basis of cohomology of $X$. 

\end{enumerate}

Let $K:=\Lambda[[(q_{\gamma,i})_{\gamma,i}]] $ denote the coefficient field of $\Phi$. We let $\mathcal{E}$ denote the universal sheaf over $\mathcal{M}^{\alpha}_{X/C}(r,L,\Delta)$. We assume that $\alpha$ is generic and that it satisfies \eqref{eq_nonemptycondition}.

If $X\to C$ has no marked fiber (so either $C$ is smooth or has two components and a single node), we define an invariant simply as
\[I^{\Phi}_{X/C}(r,L,\Delta):=\int \Phi(\mathcal{E})\cap [M^{\alpha}_{X/C}(r,L,\Delta)]^{\vir}\in K.\]
Here, the left hand side a-priori implicitly depends on $\alpha$, but it will follow from the decomposition formula that it is actually independent for any generic $\alpha$ satisfying \eqref{eq_nonemptycondition}.

If $C$ is smooth with a single marked point, we let $L_F$ denote the restriction of $L$ to the marked fiber, and set
\[V:= H_*(M_F(r,L_F), K).\]
Then we obtain invariants valued in $V$ by pushing forward along the evaluation map
\[I^{\Phi}_{X/C}(r,L,\Delta):=\operatorname{ev}_* \left(\Phi(\mathcal{E})\cap [M^{\alpha}_{X/C}(r,L,\Delta)]^{\vir}\right)\in V.\]

If $C$ has a marked point, let $F$ be the marked fiber. Otherwise, let $F$ denote the fiber over an arbitrary smooth point of $C$. We set

\[Z_{X/C,\Phi}(q) := \sum_{\substack{\Delta\in \ZZ\\0\leq \ell<r}} I_{X/C}^{\Phi}(r,L+ \ell[F],\Delta)\,  q^{\Delta - (r^2-1)\chi(\mathcal{O}_X)}.\]
This is valued in $K[[q]]$ or $V[[q]]$ respectively and depends implicitly on $r,L$ and $L_0$. Here, we choose a different generic $\alpha$ satisfying \eqref{eq_nonemptycondition} for each $\Delta$ and $\ell$. 

\begin{rem}\label{rem_coeffs}
Note that for $c_1= L+\ell F$, the discrimant and second chern class are related by 
\[\Delta = 2 rc_2 - (r-1) c_1^2 = 2 r c_2 - (r-1) L^{\cdot 2} - 2 \ell (r-1) d.\]
 As $c_2$ ranges through the integers and $\ell$ ranges in $[0,r-1]$, we find that the possible exponents of $q$ for which the coefficient of $Z_{X/C,\Phi}$ is non-empty, lie in $2\ZZ+ (r-1)L^{\cdot 2} -(r^2-1)\chi(\mathcal{O}_X)$, and each such integer corresponds to a unique choice of $\ell$ and $c_2$. In other words, we have  
\begin{gather*}
	Z_{X/C,\Phi}(q) = \qquad \qquad \qquad \qquad \qquad \qquad \qquad \\
	q^{-(r-1)L^{\cdot 2}-(r^2-1)\chi(\mathcal{O}_X)}\sum_{\substack{c_2\in \ZZ\\0\leq \ell < r}} I_{X/C}^{\Phi}(r,L+ \ell[F],\Delta(c_2,\ell))\,  q^{2 (r c_2 - \ell (r-1) d)}.
\end{gather*}
We see that the exponents of $q$ in the sum range exactly through $2\ZZ$.
\end{rem}

\subsection{Degeneration formula for multiplicative classes}
Let $X\to C$ be a fibered surface, where $C$ is the union of two smooth curves along a single node and let the set-up be as in \S \ref{subsec_invs}. 
We let $Y_i\to D_i$ be the surfaces fibered over a smooth curve with a single marking obtained as the components of $X\to C$ and let $F$ denote the fiber of $X$ over the node.
 
Let
\[V:= H_*(M_F(r,L),K)\]
so that the invariants associated to $Y_i/D_i$ for $i=1,2$ are valued in $V$.
Note that $V$ has a ring structure with respect to intersection product, which is commutative as cohomology in odd degrees vanishes. We denote the intersection product of cycles by $\alpha\cdot \beta$. We define a bilinear pairing on $V$ as 
\begin{align*}
	*_{\Phi} :\, V\times V &\to K \\
	\alpha , \beta &\mapsto \int_{M_F(r,L)} A(T_{M_F(r,L)})^{-1}\cap \alpha\cdot\beta 
\end{align*}

We extend the multiplication $*_{\Phi}$ to Laurent series in $q$ over $V$ and $K$ by applying it coefficientwise and dividing the final result by $q^{\dim M_F(r,L)}$.
\begin{thm}\label{thm_decform}
	\[	Z_{X/C,\Phi}(q) = Z_{X_1/C_1, \Phi}(q) *_{\Phi} Z_{X_2/C_2,\Phi}(q)\]
\end{thm}
\begin{proof}
Comparing coefficients, and in view of Remark \ref{rem_coeffs}, we may consider $c_2$ and $k$, and therefore $\Delta$ as fixed, and we need to show that -- for some fixed chosen stability condition $\alpha$ on $C$ -- we have
\[I^{\Phi}_{X/C}(r,L +\ell[F], \Delta) = \sum I^{\Phi}_{Y_1/D_1}(r,L_1 + \ell_1 [F], \Delta_1) *_\Phi I^{\Phi}_{Y_2/D_2}(r,L_2 + \ell_2 [F], \Delta_2),\]
where the sum ranges over all $0\leq \ell_1,\ell_2<r$ and over all $\Delta_1,\Delta_2\in \ZZ$ such that 
\[\Delta-(r^2-1)\chi(\mathcal{O}_X) = \Delta_1 +\Delta_2 - (r^2-1)(\chi(\mathcal{O}_{Y_1})+\chi(\mathcal{O}_{Y_2})) - \dim M_F(r,L_F), \]
or equivalently, such that $\Delta = \Delta_1+\Delta_2$.
Writing out the definition of invariants, we have the equivalent formula
\begin{equation}\label{eq_newform}
\begin{gathered}
\int \Phi(\mathcal{E})\cap [M_{X/C}^{\alpha}(r,L + \ell[F],\Delta)]^{\vir} =\\ 
\sum\left(\operatorname{ev}_{*}(\Phi(\mathcal{E}_1)\cap [M_{Y_1/D_1}^{b}(r, L_1+\ell_1 F, \Delta_1)]^{\vir})\, *_{\Phi}\right.\\ 
\left.\operatorname{ev}_{*}(\Phi(\mathcal{E}_2)\cap [M_{Y_2/D_2}^{b}(r, L_2+\ell_2 F, \Delta_2)]^{\vir})\right)
\end{gathered}
\end{equation}
for some choice of generic $\alpha$ satisfying \eqref{eq_nonemptycondition} and where we use the notation of Remark \ref{rem_alphas}. 

Since $\Delta = \Delta_1+\Delta_2$ and by the dependence of the discriminant on first and second Chern classes, we have that each term of the sum for which the moduli spaces are non-empty, that $\ell \equiv \ell_1 + \ell_2 \mod r$. In particular, for each such term there exists unique representatives $\ell_i' \equiv \ell_i \mod r$, such that $|\alpha(c_1(L_i+\ell_i'F), \Delta_i)-\alpha_i|<1/2$. It follows that $\ell =\ell_1' + \ell_2'$. Letting $L_i':= L_i+\ell_i' F$ and $L':= L+\ell F$, we in particular have $L_1'+ L_2' = \overline{L'}$ in the notation preceding Proposition \ref{prop_fixdetfundclass}. Since twisting by a line bundle induces an isomorphism between moduli spaces, we have 
\begin{align*}
I^{\Phi}_{Y_i/D_i}(r,L_1 + \ell_1 F, \Delta_1)
&= I^{\Phi}_{Y_i/D_i}(r,L_1 + \ell_1'F, \Delta_1) \\
&= \operatorname{ev}_*\left(\Phi(\mathcal{E}_i)\cap [M^{\rough{\alpha_i}}_{Y_i/D_i}(r, L_1',\Delta_i)]^{\vir} \right) 
\end{align*}

In summary, we may rewrite the right hand side of \eqref{eq_newform} as 
\begin{gather*}
\sum_{\substack{\Delta_1+ \Delta_2 = \Delta\\ L_1' + L_2' = \overline{L'}}}  \left(\operatorname{ev}_*(\Phi(\mathcal{E}_1)\cap [M_{Y_1/D_1}^{\rough{\alpha_1}}(r,L_1',\Delta_1)]^{\vir}) *_\Phi \right. \\
\qquad \qquad\qquad\left. \operatorname{ev}_*(\Phi(\mathcal{E}_2)\cap [M_{Y_2/D_2}^{\rough{\alpha_2}}(r,L_2',\Delta_2)]^{\vir})\right) 
\end{gather*}
We examine each term of this sum. 
Using the definition of $*_\Phi$, the projection formula and Proposition \ref{prop_virclass}, we may rewrite a single term in this sum as the pushforward to a point of
\begin{gather*}
	\operatorname{pr}_F^*A(T_{M_F})^{-1}\cap \operatorname{pr}_1^*\Phi(\mathcal{E}_1)\cap\operatorname{pr}_2^*\Phi(\mathcal{E}_2)\cap \\ \qquad \qquad \qquad [M_{Y_1/D_1}^{\rough{\alpha_1}}(r,L_1',\Delta_1)\times_{M_F} M_{Y_2/D_2}^{\rough{\alpha_2}}(r,L_2',\Delta_2) ]^{\vir}.
\end{gather*}
Then by Lemmas \ref{lem_split1} and \ref{lem_split2}, this is equal to 
\[\Gamma^*\Phi(\mathcal{E})\cap [M_{Y_1/D_1}^{\rough{\alpha_1}}(r,L_1',\Delta_1)\times_{M_F} M_{Y_2/D_2}^{\rough{\alpha_2}}(r,L_2',\Delta_2) ]^{\vir},\]
where $\Gamma$ is the glueing map to $M^{\alpha}_{X/C}(r,L', \Delta)$. Using this, we have reduced equation \eqref{eq_newform} to Proposition \ref{prop_virdec} and are done. 
\end{proof}

\subsection{Application to Elliptic Fibrations}\label{subsec_ellfib}
Suppose that $X\to C$ has genus one fibers in the situation of Theorem \ref{thm_decform}. In this case, we have $V=H_*(M_F(r,L),K)  = H_*(\operatorname{pt},K)\simeq K$, so the relative invariants are simply power series valued in the coefficient ring. In this case, the statment of Theorem \ref{thm_decform} becomes especially simple

\begin{cor}\label{cor_elldec}
	If $g=1$, we have the following identity in $K[[q]]$:
	\[Z_{X/C,\Phi}(q) = Z_{Y_1/D_1,\Phi}(q)\,Z_{Y_2/D_2,\Phi}(q).\]
\end{cor}
In the rest of this subsection, we give a prove of Theorem \ref{thm_ellgen}. We will consider the special case 
\[\Phi(\mathcal{E}) = A_{y,q}(T^{\vir})\]
 of \eqref{eq:insertions} obtained by setting $B=1$, and taking $A_{y,q}$ to be the insertion considered in \cite[\S 4.8]{GL_Toric} which defines the virtual elliptic genus \cite{FG_RR}. 
By results of de Jong and Friedman, we have suitable degenerations:
\begin{thm}\label{thm_enoughdefs}
	Let $X$ be an elliptic surface over $\PP^1$ of degree $e\geq 2$ without multiple or reduced fibers. Let $D$ be a $d$-section on $X$, and suppose that there exist no $d'$-sections for any $1\leq d'<d$. Then there exists a connected base $B$ and a family of elliptic surfaces $X_B \to C_B\to B$ together with a family of $n$-sections $D_B\subset X_B$ such that
	\begin{enumerate}[label = \roman*)]
		\item For some $b_0\in B$, the triple $D_{b_0}\subset X_{b_0}\to C_{b_0}$ is isomorphic to $D\subset  X \to \PP^1$.
		\item For some $b_1\in B$, we have: 
		\begin{itemize}
			\item $X_{b_1}\to C_{b_1}$ is obtained from glueing two elliptic surfaces $Y_1\to \PP^1$ and $Y_2\to \PP^1$ along an isomorphic fiber, where $Y_1$ is a degree $e-1$ elliptic surface and $Y_2$ is a rational elliptic surface.
			\item  The divisor $D_{b_1}\subset X_{b_1}$ restricts to a $d$-section on $Y_1$ and to a smooth rational curve $D$ satisfying $D^2 = d-2$ on $Y_2$. Moreover, if $e\geq 2$, then $Y_1$ has no $d'$-sections for $1\leq d'<  d$. 
		\end{itemize}
	\end{enumerate}
\end{thm}
\begin{proof}
	For $e\geq 3$, this follows from Theorem 4.9 in \cite{dJFr_On} together with constructions going into Claim 5.7 in \cite{dJFr_On}. For $e=2$, it follows from similar arguments using the Torelli theorem for lattice polarized K3 surfaces.
\end{proof}

We will also need the following vanishing result
\begin{prop}\label{prop_trivfactor}
Let $E$ be an elliptic curve and consider $X = E\times \PP^1 \to \mathbb{P}^1$, and let $x_1,\ldots,x_n$ be distinct points on $\mathbb{P}^1$. Let $r>0$ and let $L$ be a line bundle on $E\times \PP^1$ that has degree $d$ on fibers, with $d$ coprime to $r$. Then we have
\[Z^{\Ell}_{X/(\mathbb{P}^1, (x_1,\ldots,x_n))} = 1.\]
\end{prop}
\begin{proof}
For any $\Delta$ for which the moduli space 
\[M_{X/(\PP^1,(x_1,\ldots,x_n))}(r,L,\Delta)\]
is non-empty one can show that either it is a point, or it admits -- up to a finite \'etale cover -- an elliptic curve factor. In the latter case, all virtual Chern numbers vanish. The result follows from this.
\end{proof}

\begin{proof}[Proof of Theorem \ref{thm_ellgen}]
Since enumerative invariants for surfaces with $p_g(X)>0$ are independent of choice of polarization, we may use Theorem \ref{thm_boundedwalls}, and compute invariants using moduli spaces of $f$-stable sheaves.  

In the case $e=2$, the result follows from the DMVV formula and the fact that any moduli space of Gieseker-stable sheaves on a K3 surface is deformation invariant to a Hilbert scheme of points when stability equals semi-stability. 
Next, we show that for any rational elliptic surface $Y\to \mathbb{P}^1$ and any divisor $D$ on $Y$ of fiber degree coprime to $r$, we have (when taking invariants and generating series with respect to the moduli spaces of fiber-stable objects)
\[Z_{Y/\PP^1}^{\Ell} = (Z_{K3}^{\Ell})^{1/2}.\]
By Proposition \ref{prop_trivfactor}, the relative and absolute invariants agree. In particular, we may glue two identical copies of $Y$ together along a fiber and deform the resulting surface to a smooth K3 surface, while preserving the divisor class, see \cite[Theorem 5.10]{Fri_Glo} and \cite[Proposition 4.3]{Fri_New}.

Now, we can argue inductively on $e$, with base case $e=2$: By Theorem \ref{thm_enoughdefs} and Corollary \ref{cor_elldec}, to obtain an identity 
\[Z_{X/\PP^1}^{\Ell} = Z^{\Ell}_{X'/(\PP^1, 0)}\, Z^{\Ell}_{Y,(\PP^1,0)},\]
where $X'$ is an elliptic surface of degree $e-1$ with a chosen $d$-section $D'$ (and which possesses) no $d'$-sections for $1\leq d'< d$, and where $Y$ is a rational elliptic surface with a chosen rational curve $D$ satisfying $D^{\cdot 2} = d-2$. 
Using Proposition \ref{prop_trivfactor} again and by the inductive hypothesis
\[Z_{X/\PP^1}^{\Ell} = Z^{\Ell}_{X'/\PP^1}\, Z^{\Ell}_{Y,\PP^1} = (Z_{K3}^{\Ell})^{(e-1)/2}(Z_{K3}^{\ell})^{1/2}.\]
This finishes the proof.

\end{proof}

\bibliography{references}
\bibliographystyle{amsalpha}

\end{document}